\documentclass[pdflatex,sn-mathphys-num]{sn-jnl}

\usepackage{graphicx}%
\usepackage{multirow}%
\usepackage{amsmath,amssymb,amsfonts}%
\usepackage{amsthm}%
\usepackage{mathrsfs}%
\usepackage[title]{appendix}%
\usepackage{xcolor}%
\usepackage{textcomp}%
\usepackage{manyfoot}%
\usepackage{booktabs}%
\usepackage{algorithm}%
\usepackage{algorithmicx}%
\usepackage{algpseudocode}%
\usepackage{listings}%
\usepackage{enumitem}

\theoremstyle{thmstyleone}%
\newtheorem{theorem}{Theorem}
\newtheorem{Lemma}[theorem]{Lemma}%
\newtheorem{Remarks}[theorem]{Remarks}%
\newtheorem{thm}[theorem]{Theorem}%
\newtheorem{Cor}[theorem]{Corollary}%
\newtheorem{Prop}[theorem]{Proposition}%
\newtheorem{rem}[theorem]{Remark}%

\newcommand{\N}{\mathbb{N}}
\newcommand{\E}{\mathbb{E}}
\newcommand{\calI}{\mathcal{I}}

\newcommand{\DD}{\mathbf{D}}

\newcommand{\TT}{\mathbf{T}}
\renewcommand{\SS}{\mathbf{S}}

\renewcommand{\P}{\mathbb{P}}
\newcommand{\e}{\mathbf{e}}

\begin{document}

\title[Central Limit Theorems for Random Walks on Free Products]{Central Limit Theorems for Drift and Entropy of Random Walks on Free Products}


\author*[1]{\fnm{Lorenz A.} \sur{Gilch}}\email{Lorenz.Gilch@freenet.de}

\affil*[1]{\orgname{University of Passau}, \orgaddress{\street{Innstr. 33}, \city{94032 Passau}, \country{Germany}}}


\abstract{In this article we consider a natural class of random walks on free products of graphs, which arise as convex combinations of random walks on the single factors. From the works of Gilch \cite{gilch:07,gilch:11} it is well-known that for these random walks the asymptotic entropy as well as the drift w.r.t. the natural transition graph distance and also w.r.t. the word length exist. The aim of this article is to formulate three central limit theorems with respect to both drift definitions and also w.r.t. the entropy when seen as the limit of last visit generating functions. In the case that the random walk depends on finitely many parameters we show that the corresponding variances in the central limit theorems  w.r.t. both drifts vary real-analytically in terms of  these parameters, while the variance in the central limit theorem w.r.t. the entropy varies real-analytically in the case of free products of finite graphs.}

\keywords{free product, random walk, central limit theorem, drift, entropy, renewal times, analyticity}


\pacs[MSC Classification]{60J10, 60F05}

\maketitle

\section{Introduction}

Let $V_1, V_2$ be finite or countable, disjoint sets with $|V_i|\geq 2$, and fix distinguished elements $o_i \in V_i$, $i\in \{1,2\}$.
Suppose that each $V_i$ carries a transition matrix \mbox{$P_i \in [0,1]^{V_i \times V_i}$.}
The free product $V := V_1 \ast V_2$ consists of all finite words over the alphabet $(V_1 \cup V_2)\setminus \{o_1,o_2\}$ such that two consecutive letters do not lie in the same $V_i$.
We consider a transient Markov chain $(X_n)_{n \in\N_0}$ on $V$, starting at the empty word $X_0 := o$, with transition law given by a convex combination of the kernels $P_i$. Denote by $\Vert x\Vert$ the word length of $x\in V$ and let $d(o,x)$ denote the distance of $x$ to $o$ w.r.t. the natural graph metric on the transition graph of the random walk on $V$. Then, by \mbox{Gilch \cite{gilch:07},} it is well-known that the (asymptotic) drifts
$$
\lim_{n\to\infty} \frac{\Vert X_n\Vert}{n} \quad \textrm{ and } \quad
\lim_{n\to\infty} \frac{d(o,X_n)}{n}
$$
exist almost surely and both limits are  almost surely constant. Moreover, let $\pi_n$ denote the distribution of $X_n$. In Gilch \cite{gilch:11} it is shown that the (asymptotic) entropy
$$
\lim_{n\to\infty} \frac1{n}\mathbb{E}\bigl[-\log \pi_n(X_n)\bigr] =\lim_{n\to\infty} -\frac1{n} L(o,X_n|1) =\liminf_{n\to\infty}-\frac1n \log \pi_n(X_n)
$$
exists almost surely and the limit is also almost surely constant, where $L(o,X_n|1)$ is the last visit generating function (see Subsection \ref{subsec:random-walks}). The aim of this note is to establish three central limit theorems w.r.t.  both drift definitions and also w.r.t. the entropy when seen as the pathwise limit of $-\frac1{n} L(o,X_n|1)$. In particular, we will show that if the random walk's law depends on finitely many parameters, then the variances in these central limit theorems vary real-analytically in terms of these  parameters, where analyticity of the variance in the entropy case is restricted  to free products of finite graphs.
\par
Let me outline some results on random walks on free products and their importance.
Random walks on free products have been studied extensively, and there is a broad range of literature.
Asymptotics of return probabilities were elaborated in Gerl and Woess \cite{gerl-woess}, Woess \cite{woess3}, Sawyer \cite{sawyer78}, Cartwright and Soardi \cite{cartwright-soardi}, Lalley \cite{lalley93,lalley:04}, and Candellero and Gilch \cite{candellero-gilch}; explicit drift and entropy formulas for free products of finite groups were obtained by Mairesse and Math\'eus \cite{mairesse1}, while Gilch \cite{gilch:07,gilch:11} calculated later explicit formulas  for the drift and the entropy in the more general, inhomogeneous setting of free products of graphs. 
Spectral properties were analyzed in Shi et al. \cite{sidoravicius:18}.
\par
The relevance of free products in group theory arises from Stallings’ Splitting Theorem (see Stallings \cite{stallings:71}), which characterizes finitely generated groups with more than one geometric end as those admitting a representation as a nontrivial free product by amalgamation or as an HNN extension over a finite subgroup; we note that free products arise as the special case of amalgams over the trivial subgroup.
Most prior works focus on  free products of groups, which are space-homogeneous and admit transitive random walks.
In this article we address the more general setting of free products of graphs, which are lacking such a homogeneity.
In particular, we revisit results of \cite{gilch:07,gilch:11} on the drift and the entropy of random walks on free products of graphs, where the random walk's trajectory is decomposed into disjoint  segments via so-called exit  times $(\e_k)_{k\in\N}$. We will use this decomposition in order to derive three central limit theorems in association with the asymp\-totic drift and entropy of $(X_n)_{n\in\N_0}$. With the help of a very detailed understanding of this exit time process, the hard main work of the proofs lies in the rigorous preparation of a setting such that some standard reasoning for deriving central limit theorems can be applied.
\par
Regularity of asymptotic characteristic numbers of random walks on several classes of groups (and beyond) has attracted a lot of attention during the last years. The question of regularity goes back to Erschler and Kaimanovich who asked whether drift and entropy of random walks on groups vary continuously or even analytically when the support of
single step transition probabilities is kept constantly. This question has been investigated extensively in the last years, amongst others, by Ledrappier \cite{ledrappier12,ledrappier12-2}, Mathieu \cite{mathieu:15}, and Gilch \cite{gilch:07,gilch:11,gilch:15}. In the setting of surface groups, Ha\"issinsky et al. \cite{haissinsky-mathieu-mueller12} established the analyticity of the drift for random walks and, in addition, proved a central limit theorem for the word length. Gouëzel \cite{gouezel:15} showed that, for random walks on hyperbolic groups, the rate of escape with respect to certain word metrics, the asymptotic variance, and the asymptotic entropy all depend real-analytically on the underlying random walk.
The survey article of Gilch and Ledrappier \cite{gilch-ledrappier13} collects several results on analyticity of drift and entropy of random walks on groups. 
More recently, Choi \cite{choi:26} discussed differentiability of the drift for random walks on relatively hyperbolic groups, while Silva \cite{silva:26} proved continuity of asymptotic entropy for non-degenerate random walks for a class of wreath products. In the current paper we will prove analyticity of the variances in the proposed central limit theorems, that is, these variances can be expanded as multivariate power series in finitely many variables which describe the single step transition probabilities. This extends the recent research within the topic of regularity of random walk quantities.
\par
The paper is organized as follows: Section \ref{sec:free-products} recalls the definitions of free products and associated random walks on them.
In Section \ref{sec:last-cone-entry-process} we introduce last cone entry times (exit times) and derive a sequence of renewal times $(\TT_k)_{k\in\N_0}$. Finally, in
Sections \ref{sec:CLT}, \ref{sec:CLT-2} and \ref{sec:CLT-3} we derive the proposed central limit theorems and show that the variances vary real-analytically, when the random walk's probability law depends on finitely many parameters only. In the Appendix \ref{app:proof-Dk-iid} 
we outsource some proofs for better fluidity of reading.

\section{Free Products and Random Walks}
\label{sec:free-products}

\subsection{Free Products}

Let $V_1, V_2$ be disjoint, finite or countable sets with $|V_i|\geq 2$ for $i\in\mathcal{I}:=\{1,2\}$. We fix a distinguished element $o_i\in V_i$, called the \textit{root} of $V_i$, 
for each $i\in\mathcal{I}$, and we set $V_i^\times:=V_i\setminus\{o_i\}$.  The  case $|V_1|=|V_2|=2$ will be excluded; see Remark \ref{rem:case-2-2}.(i).  

The \textit{free product} of $V_1$ and $V_2$ is  given by the set
\[
V := V_1 \ast V_2 := \Bigl\{ v_1\ldots v_k \,\Bigl|\,  k\in\mathbb{N}, \; v_\ell \in V_1^\times\cup V_2^\times,\; v_i\in V_j^\times \Rightarrow v_{i+1}\notin V_j^\times \Bigr\}\cup\{o\},
\]
the set of finite words over the alphabet $V_1^\times\cup V_2^\times$ with no two consecutive letters coming from the same $V_j^\times$, where $o$ denotes the empty word.  
\par
A partial composition law is defined as follows: if we have $u=u_1\ldots u_m$, \mbox{$v=v_1\ldots v_n\in V$,} $m,n\in\N$, and $u_m\in V_i^\times$, $v_1\notin V_i^\times$ for some $i\in\mathcal{I}$, then $uv$ is the concatenated word, which is again an element of $V$.  
We write $\delta(u):=i$, and make the convention $uo_j=u=ou=uo$ for $j\in\mathcal{I}\setminus\{\delta(u)\}$.  Furthermore,  the \textit{word length} of $u=u_1\ldots u_m$ is denoted by $\Vert u\Vert :=m$.  
\par
For $u\in V$, the \textit{cone} rooted at $u$ is  given by
\[
C(u):=\bigl\{w\in V \,\bigl|\, w \text{ has prefix } u\bigr\},
\]
the set of words in $V$ starting with $u$.

\subsection{Random Walks}
\label{subsec:random-walks}
Let $P_i=\bigl(p_i(x,y)\bigr)_{x,y\in V_i}$ be transition matrices on $V_i$, $i\in\mathcal{I}$, such that every $x\in V_i^\times$ is accessible from $o_i$ in finitely many steps with positive probability, that is, \mbox{$P_i^{n_x}(o_i,x)>0$} for some $n_x\in\N$.  For sake of simplicity, we assume that $p_i(x,x)=0$ for every $i\in\mathcal{I}$ and all $x\in V_i$, and we assume that there exist $j\in\mathcal{I}$, $y\in V_j^\times$ and $n_y\in \N$ with $P_j^{n_y}(y,y)>0$; see Remarks \ref{rem:case-2-2}.(iii).
Fix $\alpha\in(0,1)$, and set $\alpha_1:=\alpha$, $\alpha_2:=1-\alpha$.  
Define the lifted transition matrix $\overline P_i=\bigl(\overline p_i(x,y)\bigr)_{x,y\in V}$ by  
\[
\overline p_i(uv,uw):=p_i(v,w),
\]
if $u=u_1\ldots u_m$ with $u_m\notin V_i$ and $v,w\in V_i$, $i\in\mathcal{I}$; otherwise, we set $\overline p_i(x,y):=0$.
A natural random walk on $V$ is governed by the transition matrix   
\[
P = \bigl(p(x,y)\bigr)_{x,y\in V} := \alpha_1\overline P_1+\alpha_2\overline P_2.
\]
The transition graph $\mathcal{X}$ w.r.t. $P$ arises from the transition graphs $\mathcal{X}_1, \mathcal{X}_2$ w.r.t. $P_1$, $P_2$ in a natural way as follows: take copies of $\mathcal{X}_1$ and $\mathcal{X}_2$ and glue them together at their roots $o_1$ and $o_2$, which becomes the single vertex $o$. Inductively, at each in the previous step newly added vertex \mbox{$u=u_1\ldots u_m\in V\setminus\{o\}$} with $\delta(u)=i\in\calI$  attach a copy of $\mathcal{X}_j$, $j\in\mathcal{I}\setminus\{\delta(u)\}$, where $u$ and the root $o_j$ of the new copy of $\mathcal{X}_j$ are glued together to one single vertex, which becomes  $u=uo_j$. The vertices $u_{m+1}\in V_j^\times$ of the newly attached copy $\mathcal{X}_j$ become the elements $u_1\ldots u_mu_{m+1}$ in $\mathcal{X}$. See, e.g., Gilch \cite[Example 2.1]{gilch:22} for a graphic visualization.
\par
The transition graph $\mathcal{X}$ gives rise to a natural graph distance $d(\cdot,\cdot)$; that is, for $x,y\in V$, $d(x,y)$ denotes the minimal length of an (oriented) path from $x$ to $y$ in $\mathcal{X}$.
\par
Let $(X_n)_{n\in\N_0}$ describe a random walk on $V$ governed by $P$ with $X_0:=o$. 
The corresponding $n$-step transition probabilities are  denoted by
\[
p^{(n)}(x,y):=\mathbb{P}[X_n=y\mid X_0=x], \qquad x,y\in V,\; n\in\mathbb{N}_0.
\]
Additionally, we write $\mathbb{P}_x[\,\cdot\,]:=\mathbb{P}[\, \cdot\mid X_0=x]$.  
\par
We say that the \textit{random walk's law depends on finitely many parameters} if there are $d\in\N$ and $p_1,\ldots, p_d\in (0,1)$ such that for all $x,y\in V$ with $p(x,y)>0$ there exists $k\in\{1,\ldots,d\}$ such that $p(x,y)=p_k$. In other words, $P$ has only the entries $p_1,\ldots, p_d,0$ and $1$. 
In this case write $E_i :=\bigl\{(x ,y ) \in V_i^2 \,\bigl|\, p_i(x,y)>0\bigr\}$, $i\in\mathcal{I}$, for the set of edges of $\mathcal{X}_i$ and denote by $\eta:E_1\cup E_2 \to \{p_1,\ldots,p_d\}$ the corresponding mapping $\eta(x,y)=p(x,y)$. We say that a parameter vector $\underline{p} :=(p_1,\ldots, p_d ) \in(0,1)^d$ \textit{permits a well-defined random walk on $V$} if
\begin{eqnarray*}
\forall x_1 \in V_1: &&  \sum_{y_1\in V_1: (x_1,y_1)\in E_1} \eta(x_1,y_1) + \sum_{y_2\in V_2: (o_2,y_2)\in E_2} \eta(o_2,y_2) = 1,\\
\forall x_2 \in V_2: &&  \sum_{y_1\in V_1: (o_1,y_1)\in E_1} \eta(o_1,y_1) + \sum_{y_2\in V_2: (x_2,y_2)\in E_2} \eta(x_2,y_2) = 1.
\end{eqnarray*}
In this case we have $\alpha= \sum_{y\in V_1: (o_1,y)\in E_1}  \eta(o_1,y)$, $p_1(x_1,y_1) = \eta(x_1,y_1)/\alpha$ and $p_2(x_2,y_2) = \eta(x_2,y_2)/(1-\alpha)$ for $(x_1,y_1) \in E_1, (x_2,y_2) \in E_2$, which leads to a stochastic matrix $P$ on $V$ and therefore defines a well-defined random walk on the free product. Denote by
$$
\mathcal{P}_d :=\bigl\{\underline{p} :=(p_1,\ldots,p_d) \in (0,1)^d \,\bigl|\, \textrm{$\underline{p}$ defines a well-defined random walk on $V$} \bigr\}
$$
the set of parameter vectors which allow well-defined random walks.
%
%
%
\par
An important property is given by the following lemma which states that paths inside a cone can be shifted to paths originating from $o$ in a measure preserving way:
\begin{Lemma}\label{lem:cone-probabilities}
Let be $n\in\mathbb{N}$, $x\in V$, and $v_1,\dots,v_n \in C(x)$. We write $v_i=xu_i$, where $u_i\in V$ does not start with a letter in $V_{\delta(x)}^\times$. Then
\[
\mathbb{P}_x[X_1=v_1,\dots,X_n=v_n]=\mathbb{P}_o[X_1=u_1,\dots,X_n=u_n].
\]
\end{Lemma}
\begin{proof}
See \cite[Lemma 3.2]{gilch:22}.
\end{proof}

We introduce the following generating functions: for $x,y\in V$, $z\in\mathbb{C}$, $i\in\mathcal{I}$, define
\[
\begin{aligned}
G(x,y|z) &:= \sum_{n\geq 0} p^{(n)}(x,y)\cdot z^n \quad \quad \textrm{(Green function)},\\
L(x,y|z) &:= \sum_{n\geq 0} \mathbb{P}_x\bigl[ X_n=y,  \forall m\in\{1,\ldots,n\}: X_m \neq x \bigr]\cdot z^n\quad \textrm{(last visit gen. fct.)},\\
\xi_i(z) &:= \sum_{n\geq 1} \mathbb{P}\bigl[X_n\in V_i^\times, \forall m<n: X_m\notin V_i^\times \bigr]\cdot z^n.
\end{aligned}
\]
Analogously, for $x_i,y_i\in V_i$, let $G_i(x_i,y_i|z)$ and $L_i(x_i,y_i|z)$ denote the corresponding functions on $V_i$ with respect to $P_i$.  
\par
In the following we collect some important equations in relation with these generating functions.
By \cite[Lemma~1.6]{gilch},  we have
\begin{eqnarray}
G(x,y|z) &=& G(x,x|z)\cdot L(x,y|z),\label{eq:G-L} \\
G_i(x_i,y_i|z) &=& G_i(x_i,x_i|z)\cdot L_i(x_i,y_i|z) \quad \textrm{ for } i\in\mathcal{I}, x_i,y_i\in V_i.\label{eq:Gi-Li}
\end{eqnarray}
If every path from $x$ to $y$ has to pass through $w$, then  
\begin{equation}\label{eq:L-L}
L(x,y|z)=L(x,w|z)\cdot L(w,y|z).
\end{equation}
Moreover, by \cite[Proposition 2.7]{gilch}, for $x_i,y_i\in V_i\subset V$,
\begin{equation}\label{eq:G-xi}
L(x_i,y_i|z)=L_i\bigl(x_i,y_i \,\bigl| \,\xi_i(z)\bigr).
\end{equation}
We remark that, for $i\in\mathcal{I}$ and $z>0$,
\begin{equation}\label{lem:4-1equ-1}
 \sum_{x\in V_i^\times} L_i\bigl(o_i,x\,\bigl|\, \xi_i(z)\bigr) \stackrel{(\ref{eq:Gi-Li})}{=} \sum_{x\in V_i^\times} \frac{G_i\bigl(o_i,x\,\bigl|\, \xi_i(z)\bigr)}{G_i\bigl(o_i,o_i\,\bigl|\, \xi_i(z)\bigr)}
\end{equation}
and
\begin{eqnarray}
\sum_{x\in V_i^\times}G_i\bigl(o_i,x\,\bigl|\, \xi_i(z)\bigr)&\leq& \sum_{x\in V_i}\sum_{n\geq 0} p_i^{(n)}(o_i,x)\cdot \xi_i(z)^n \nonumber\\
&=&\sum_{n\geq 0}  \underbrace{\sum_{x\in V_i}p_i^{(n)}(o_i,x)}_{=1} \cdot \xi_i(z)^n
 =\frac{1}{1-\xi_i(z)}.\label{lem:4-1equ-2}
\end{eqnarray}

For $x\in V\setminus\{o\}$ with $\delta(x)=i$, set  
\[
\xi_i:=\xi_i(1)=\mathbb{P}_x\bigl[\exists n\in\mathbb{N}: X_n\notin C(x)\bigr],
\]
which depends only on $i$ but not on $x$ itself due to the recursive structure of free products.  
By \cite[Lemma~2.3]{gilch:07}, we have $\xi_i<1$, hence
\[
\mathbb{P}_x\bigl[ \forall n\in\mathbb{N}: X_n\in C(x)\bigr]=1-\xi_i>0.
\]

As a \textit{basic assumption} we assume throughout this article that the spectral radius at $o$ satisfies  
\[
\varrho:=\limsup_{n\to\infty}p^{(n)}(o,o)^{1/n}<1,
\]
or equivalently, we assume that the Green function $G(o,o|z)$ has radius of convergence strictly bigger than $1$. This guarantees that all generating functions $G(x,y|z)$, $L(x,y|z)$, and $\xi_i(z)$ have radii of convergence strictly bigger than $1$; see Woess \cite[Proposition 9.18]{woess}. In particular,  $(X_n)$ is transient under this assumption.

\begin{Remarks}\label{rem:case-2-2}
\begin{enumerate}[label=(\roman*)]
\item If $|V_1|=|V_2|=2$ and $P_1,P_2$ are irreducible, then the random walk is recurrent, which is easy to check.
\item If one out of $P_1,P_2$ is not irreducible, then $\varrho<1$, which is also easy to see. 
If $P_1$ and $P_2$ are irreducible and reversible, then we have again $\varrho<1$; see \cite[Theorem 10.3]{woess}.  
\item The assumption $p_i(x,x)=0$ for every $i\in\mathcal{I}$ and $x\in V_i$ is just for presentational reason in order to avoid lengthy case distinctions, which does not affect the validity of the results at all. The assumption on existence of $j\in\mathcal{I}$, $y\in V_j^\times$ and $n_y\in\N$ with $p_j^{(n_y)}(y,y)>0$ ensures that the variances in our central limit theorems are non-zero. However, this assumption is also just for sake of simplicity and can be weakened, but it may not  be dropped completely. 
\end{enumerate}
\end{Remarks}

The \textit{rate of escape} or \textit{drift} of $(X_n)_{n\in\N_0}$ is given by the almost sure constant limit 
$$
\lambda = \lim_{n\to\infty} \frac{d(o,X_n)}{n}.
$$
In  \cite[Corollary 4.2]{gilch} it is shown that $\lambda>0$ exists and it is almost surely constant. Moreover, there exists also a constant $\ell>0$ such that
$$
\ell =\lim_{n\to\infty} \frac{\Vert X_n\Vert}{n} \quad \textrm{ almost surely};
$$
see \cite[Theorem 3.3]{gilch:07}. This limit is called \textit{rate of escape w.r.t. the block/word length}.
\par
For $\varepsilon_0>0$, we say that $P$ is \textit{$\varepsilon_0$-uniform} if, for all $x,y\in V$, 
$$
p(x, y) > 0 \quad \Longrightarrow \quad  p(x, y) \geq \varepsilon_0 >0.
$$
Denote by $\pi_n$ the distribution of $X_n$. If  $P$ is $\varepsilon_0$-uniform, then  \mbox{\cite[Theorem 3.7]{gilch:11}} guarantees existence of  a real number $h>0$ such that
$$
h=\lim_{n\to\infty} -\frac1n \mathbb{E}\bigl[\log \pi_n(X_n)\bigr]=\liminf_{n\to\infty} -\frac1n \log \pi_n(X_n)=\lim_{n\to\infty}-\frac1n \log L(o,X_n|1) \quad \textrm{ a.s.};
$$
this limit is called the \textit{asymptotic entropy} of $(X_n)_{n\in\N_0}$.
\par
The aim of this article is to prove the following three central limit theorems, where $(\TT_k)_{k\in\N_0}$ is a sequence of renewal times, which we will define in (\ref{equ:def-Tk}) in Section \ref{sec:last-cone-entry-process}.
\begin{thm}\label{th:CLT}
Assume that $G(o,o|z)$ has radius of convergence strictly bigger than $1$. Then:
$$
\frac{d(o,X_n)-n\cdot \lambda}{\sigma_\lambda\cdot \sqrt{n}} \xrightarrow{\mathcal{D}} N(0,1),
$$
where
$$
\sigma_\lambda^2=\frac{\E\Bigl[\Bigl(d(X_{\TT_0},X_{\TT_1})-(\TT_1-\TT_0)\cdot\lambda\Bigr)^2\Bigr]}{\E[\TT_1-\TT_0]}.
$$
Moreover, if the random walk's law depends on finitely many parameters $p_1,\ldots, p_d\in(0,1)$, then the mapping
$$
\mathcal{P}_d \ni (p_1,\ldots,p_d) \mapsto \sigma_\lambda^2=\sigma_\lambda^2(p_1,\ldots,p_d)
$$
varies real-analytically.
\end{thm}
Real-analyticity means that $\sigma_\lambda(p_1,\ldots,p_d)$ can be expanded as a multivariate power series in the variables $p_1,\ldots,p_d$ in a neighbourhood of any $\underline{p}_0 \in\mathcal{P}_d$. We refer to the beginning of Section 5 in \cite{gilch:22}, where more details and explanations concerning real-analyticity in terms of $p_1,\ldots,p_d$  are given.
\par
An analogous central limit theorem holds for the rate of escape w.r.t. the block length:
\begin{thm}\label{th:CLT-2}
Assume that $G(o,o|z)$ has radius of convergence strictly bigger than $1$. Then:
$$
\frac{\Vert X_n\Vert-n\cdot \ell}{\sigma_\ell\cdot \sqrt{n}} \xrightarrow{\mathcal{D}} N(0,1),
$$
where
$$
\sigma_\ell^2=\frac{\E\bigl[\bigl(2-(\TT_1-\TT_0)\cdot\ell\bigr)^2\bigr]}{\E[\TT_1-\TT_0]}.
$$
Moreover, if the random walk's law depends on finitely many parameters $p_1,\ldots, p_d\in(0,1)$, then the mapping
$$
\mathcal{P}_d \ni (p_1,\ldots,p_d) \mapsto \sigma_\ell^2=\sigma_\ell^2(p_1,\ldots,p_d)
$$
varies real-analytically.
\end{thm}

If $X_{\TT_0}=x\in V$ with $\delta(x)=1$ and $X_{\TT_1}=xy_1x_1$ with $x_1\in V_1^\times$ and $y_1\in V_2^\times$, then we set $\mathbf{W}_1:=y_1x_1$. With this notation 
we also have the following central limit theorem related with the asymptotic entropy:
\begin{thm}\label{th:CLT-3}
Assume that $P$ is $\varepsilon_0$-uniform for some $\varepsilon_0>0$ and that $G(o,o|z)$ has radius of convergence strictly bigger than $1$. Then:
$$
\frac{-\log L(o,X_n|1)-n\cdot h}{\sigma_h\cdot \sqrt{n}} \xrightarrow{\mathcal{D}} N(0,1),
$$
where
$$
\sigma_h^2=\frac{\E\Bigl[\Bigl(-\log L(o,\mathbf{W}_1|1)-(\TT_1-\TT_0)\cdot h\Bigr)^2\Bigr]}{\E[\TT_1-\TT_0]}.
$$
\end{thm}

If we consider free products of \textit{finite} graphs then the underlying transition matrix $P$ is $\varepsilon_0$-uniform for some sufficiently small $\varepsilon_0>0$. In this case 
  we also have the following analyticity result:
\begin{Cor}\label{cor:CLT-3}
Let $V_1,V_2$ be finite, and assume that $G(o,o|z)$ has radius of convergence strictly bigger than $1$. 
If the random walk's law depends on finitely many parameters \mbox{$p_1,\ldots, p_d\in(0,1)$,} then the mapping
$$
\mathcal{P}_d \ni (p_1,\ldots,p_d) \mapsto \sigma_h^2=\sigma_h^2(p_1,\ldots,p_d)
$$
varies real-analytically.
\end{Cor}

\section{Last Cone Entry Times}
\label{sec:last-cone-entry-process}

The main idea is  to decompose the random walk's trajectory into i.i.d. pieces, which allow us to derive the proposed central limit theorems. 
For this purpose, we want to track the random walk's trajectory to ``infinity'', that is, the way how $(X_n)_{n\in\N_0}$ converges to some boundary point, which we do not specify closer.
For $k\in \N$, define the \textit{$k$-th last cone entry time} (or \textit{exit time}) as
$$
\mathbf{e}_k := \inf \Bigl\{ m>0 \,\Bigl|\, \Vert X_m\Vert =k, \forall n\geq m: X_n\in C(X_{m})\Bigr\}.
$$
In other words, the random time $\mathbf{e}_k$ is the first instant of time from which on the random walk remains in the cone $C(X_{\mathbf{e}_k})$, that is, from time $\e_k$ on the first $k$ letters of $X_n$ remain unchanged. 
In particular, we have \mbox{$X_{\e_{k}-1}\notin C(X_{\e_k})$.} These last cone entry times have been used in an essential way in \cite{gilch:07,gilch:11,gilch:22}. In \cite[Proposition 2.5]{gilch:07} it is shown that $\lim_{n\to\infty}\Vert X_n\Vert =\infty$ almost surely, which implies that $\mathbf{e}_k<\infty$ almost surely for every $k\in\N$. Hence, we obtain a sequence of nested cones $C(X_{\mathbf{e}_k})\supset C(X_{\mathbf{e}_{k+1}})$, $k\in\N$, which are successively finally  entered by the random walk without any further exits; this sequence tracks the random walk's path to ``infinity''.
%
\par
In the following we  are interested just in those exit times $\e_{k}$ 
when the random walk finally enters a cone $C(X_{\e_k})$ with $\delta(X_{\e_k})=1$. 
Since the letters of a word in $V$ arise alternatingly from $V_1$ and $V_2$ 
we either must have 
$$
\delta(X_{\e_1})=\delta(X_{\e_3})=\delta(X_{\e_5})=\ldots=1 \ \textrm{ and  }\ \delta(X_{\e_2})=\delta(X_{\e_4})=\ldots=2
$$ 
\textit{or} we must have 
$$
\delta(X_{\e_1})=\delta(X_{\e_3})=\delta(X_{\e_5})=\ldots=2 \ \textrm{ and } \ \delta(X_{\e_2})=\delta(X_{\e_4})=\ldots=1.
$$
Therefore, we filter the sequence $(\e_k)_{k\in\N}$ 
accordingly as follows: 
let be 
$$
\tau := \begin{cases}
1, & \textrm{if } X_{\mathbf{e}_1}\in V_1^\times,\\
2, &\textrm{if } X_{\mathbf{e}_2}\in V_1^\times,
\end{cases}
$$
and  set for $k\in\N_0$
\begin{equation}\label{equ:def-Tk}
\TT_k:=\e_{2k+\tau}.
\end{equation}
A crucial observation will be the following proposition:

\begin{Prop}\label{prop:T1-T0-exp-mom}
$\TT_1-\TT_0$ has exponential moments.
\end{Prop}
\begin{proof}
First, we note that Lemma \ref{lem:cone-probabilities} gives for $z\in\mathbb{C}$ and every $x\in V$ with $\delta(x)=1$:
$$
\sum_{\substack{n\in\N,\\ x_1\in V_1^\times,\\ y_1\in V_2^\times}} \P_x\left[\begin{array}{c} X_{n-1}\notin C(xy_1x_1),\\ X_n=xy_1x_1, \\ \forall m<n: X_m\in C(x)\end{array}\right]\cdot z^n 
= \sum_{\substack{n\in\N,\\ x_1\in V_1^\times,\\ y_1\in V_2^\times}} \P\left[\begin{array}{c} X_{n-1}\notin C(y_1x_1),\\ X_n=y_1x_1, \\ \forall m<n: X_m\notin V_1^\times\end{array}\right]\cdot z^n. 
$$
Denote by
$$
\mathcal{W}:=V_1^\times \cup \bigl\{u_2 u_1\mid  u_1\in V_1^\times, u_2\in V_2^\times\bigr\}
$$
the support of $X_{\TT_0}$. Then we obtain for $z>0$ by decomposing according to the values of $\TT_0,\TT_1$ and $X_{\TT_0},X_{\TT_1}$:
\begin{eqnarray*}
&&\mathbb{E}\bigl[z^{\TT_1-\TT_0}\bigr] \\[1ex] 
&=&  \sum_{\substack{n\in\N,\\ x\in \mathcal{W},\\ x_1\in V_1^\times, y_1\in V_2^\times}} \P\Bigl[ X_{\TT_0}=x, X_{\TT_1}=xy_1x_1, \TT_1-\TT_0=n\Bigr]\cdot z^n\\
&=& \sum_{\substack{l\in\N, \\ x\in \mathcal{W}}}  \P\left[\begin{array}{c} X_{l-1}\notin C(x), \\ X_l=x\end{array}\right] \cdot \sum_{\substack{n\in\N,\\ x_1\in V_1^\times,\\ y_1\in V_2^\times}} \P_x\left[\begin{array}{c} X_{n-1}\notin C(xy_1x_1),\\ X_n=xy_1x_1, \\ \forall m<n: X_m\in C(x)\end{array}\right]\cdot z^n \\
&&\quad \cdot \P_{xy_1x_1}\Bigl[\forall j\geq 1: X_j\in C(xy_1x_1)\Bigr]\\[1ex]
&\stackrel{\textrm{Lemma \ref{lem:cone-probabilities}}}{=}&  \sum_{\substack{l\in\N, \\ x\in \mathcal{W}}}  \P\left[\begin{array}{c} X_{l-1}\notin C(x), \\ X_l=x\end{array}\right] \cdot \underbrace{\sum_{\substack{n\in\N,\\ x_1\in V_1^\times,\\ y_1\in V_2^\times}} \P\left[\begin{array}{c} X_{n-1}\notin C(y_1x_1),\\ X_n=y_1x_1, \\ \forall m<n: X_m\notin V_1^\times\end{array}\right]\cdot z^n}_{=:\mathcal{F}(z)} \cdot (1-\xi_1)\\
&=& \underbrace{\sum_{\substack{l\in\N, \\ x\in \mathcal{W}}}  \P\left[\begin{array}{c} X_{l-1}\notin C(x), \\ X_l=x\end{array}\right]\cdot (1-\xi_1)}_{=\P[\TT_0<\infty]=1} \cdot \mathcal{F}(z) =\mathcal{F}(z).
%
%
\end{eqnarray*}
For $z>0$,  we can bound $\mathcal{F}(z)$ from above as follows:
\begin{eqnarray*}
\mathcal{F}(z)& \leq& \sum_{\substack{x_1\in V_1^\times,\\ y_1\in V_2^\times}} G(o,y_1x_1|z) \stackrel{(\ref{eq:G-L})}{=} \sum_{\substack{x_1\in V_1^\times,\\ y_1\in V_2^\times}}G(o,o|z) \cdot L(o,y_1x_1|z) \\
&\stackrel{(\ref{eq:L-L}),(\ref{eq:G-xi})}{=}& G(o,o|z) \cdot\sum_{\substack{x_1\in V_1^\times,\\ y_1\in V_2^\times}} L_2\bigl(o_2,y_1|\xi_2(z)\bigr) \cdot L_1\bigl(o_1,x_1|\xi_1(z)\bigr) \\
&\stackrel{(\ref{lem:4-1equ-1}),(\ref{lem:4-1equ-2})}{\leq} & \frac{G(o,o|z)}{\bigl(1-\xi_1(z)\bigr)\cdot G_1\bigl(o_1,o_1|\xi_1(z)\bigr)\cdot \bigl(1-\xi_2(z)\bigr)\cdot G_2\bigl(o_2,o_2|\xi_2(z)\bigr)}.
\end{eqnarray*}
Recall that $G(o,o|z)$,  $G_i\bigl(o_i,o_i|\xi_i(z)\bigr)$ and $\xi_i(z)$ have radii of convergence strictly bigger than $1$ and that $\xi_i(1)=\xi_i<1$ for $i\in\mathcal{I}$. Therefore, continuity of the involved functions together with Pringsheim's Theorem yield that $\mathcal{F}(z)$ has radius of convergence strictly bigger than $1$, which in turn implies that $\mathbb{E}\bigl[z^{\TT_1-\TT_0}\bigr]$ has radius of convergence strictly bigger than $1$. This proves existence of exponential moments of $\TT_1-\TT_0$.
\end{proof}

We also have:

\begin{Prop}\label{lem:T0-exp-mom}
$\TT_0$ has exponential moments.
\end{Prop}
\begin{proof}
For $z>0$, we obtain by decomposing according to the values of $\TT_0$ and $X_{\TT_0}$:
\begin{eqnarray*}
&&\E\bigl[z^{\TT_0}\bigr]=\sum_{n\geq 1} \P[\TT_0=n]\cdot z^n= \sum_{n\geq 1} \P[\mathbf{e}_{\tau}=n]\cdot z^n\\
&=& \underbrace{\sum_{\substack{m\in\N,\\ x\in V_1^\times}}\P\left[\begin{array}{c} X_{m-1}\notin C(x),\\ X_m=x\end{array}\right]\cdot z^m \cdot \P_x\bigl[\forall j\geq 1: X_j\in C(x)\bigr]}_{\textrm{case $\tau=1$}}\\
&&\quad +  \underbrace{\sum_{\substack{\substack{n\in\N,\\ x_1\in V_1^\times,\\ y_1\in V_2^\times}}}\P\left[\begin{array}{c}X_{n-1}\notin C(y_1x_1),\\ X_n=y_1x_1\end{array}\right] \cdot z^n \cdot \P_{y_1x_1}\bigl[\forall j\geq 1: X_j\in C(y_1x_1)\bigr]}_{\textrm{case $\tau=2$}}\\
&\leq &  \sum_{x\in V_1^\times}G(o,x|z) \cdot (1-\xi_1) + \sum_{\substack{x_1\in V_1^\times,\\ y_1\in V_2^\times}}G(o,y_1x_1|z) \cdot (1-\xi_1)\\
&\stackrel{(\ref{eq:G-L})}{=}&  \sum_{x\in V_1^\times}G(o,o|z)\cdot L(o,x|z) \cdot (1-\xi_1) \\
&&\quad + \sum_{x_1\in V_1^\times, y_1\in V_2^\times}G(o,o|z)\cdot L(o,y_1x_1|z) \cdot (1-\xi_1)\\
&\stackrel{(\ref{eq:L-L}),(\ref{eq:G-xi})}{=}&  G(o,o|z)\cdot (1-\xi_1) \\
&&\quad \cdot \Biggl( \sum_{x\in V_1^\times} L_1\bigl(o_1,x\,\bigl|\, \xi_1(z)\bigr) + \sum_{\substack{x_1\in V_1^\times,\\ y_1\in V_2^\times}}L_2\bigl(o_2,y_1\,\bigl|\, \xi_2(z)\bigr)\cdot L_1\bigl(o_1,x_1\,\bigl|\, \xi_1(z)\bigr) \Biggr) \\
&\leq & \frac{G(o,o|z)\cdot (1-\xi_1)}{\bigl(1-\xi_1(z)\bigr)\cdot G_1\bigl(o_1,o_1\,\bigl|\, \xi_1(z)\bigr)} \cdot \biggl(1+\frac1{\bigl(1-\xi_2(z)\bigr)\cdot G_2\bigl(o_2,o_2\,\bigl|\, \xi_2(z)\bigr)}\biggr).
\end{eqnarray*}
Since $G(o,o|z)$, $G_i\bigl(o_i,o_i\,\bigl|\, \xi_i(z)\bigr)$ and $\xi_i(z)$, $i\in\mathcal{I}$, have radii of convergence strictly bigger than $1$ and $\xi_i(1)<1$, Pringsheim's Theorem yields that the power series $\sum_{n\geq 1} \P[\TT_0=n]\cdot z^n$ has also radius of convergence strictly bigger than $1$. Hence, $\TT_0$ has exponential moments.
\end{proof}
Another essential property is given by the next proposition:
\begin{Prop}\label{prop:Tk-Tk-1-iid}
$(\TT_k-\TT_{k-1})_{k\in\N}$ is an i.i.d. sequence.
\end{Prop}
\begin{proof}
Let be $k\in\N_0$, and denote the support of $X_{\TT_k}$ by 
$$
\mathcal{W}_k:=\Bigl\{x \in V \,\Bigl|\, \Vert x\Vert \in \{2k+1,2k+2\}, \delta(x)=1\Bigr\}.
$$
First, we prove that $\TT_{k}-\TT_{k-1}$, $k\in\N$, has the same distribution as $\TT_1-\TT_0$. To this end, we make a case distinction according to the values of $\TT_{k-1},\TT_{k}$ and $X_{\TT_{k-1}}$, $X_{\TT_{k}}$. For each $m\in\N$, we get
\pagebreak[4]
\begin{eqnarray*}
&& \P\bigl[\TT_{k}-\TT_{k-1}=m\bigr] \\[1ex] 
&=& \sum_{\substack{x\in\mathcal{W}_{k-1},\\ x_1\in V_1^\times,\\ y_1\in V_2^\times}}\P\Bigl[\TT_{k}-\TT_{k-1}=m,X_{\TT_{k-1}}=x,X_{\TT_{k}}=xy_1x_1\Bigr] \\
&=& \sum_{\substack{l\in\N,\\ x\in \mathcal{W}_{k-1}}} \P\left[\begin{array}{c} X_{l-1}\notin C(x),\\ X_l=x\end{array}\right] 
\cdot \sum_{\substack{x_1\in V_1^\times,\\ y_1\in V_2^\times}} \P_x\left[\begin{array}{c} X_{m-1}\notin C(xy_1x_1), \\ X_m=xy_1x_1,\\ \forall m'<m: X_{m'}\in C(x) \end{array}\right]\\
&&\quad \cdot \underbrace{\P_{xy_1x_1}\bigl[\forall j\geq 1: X_j\in C(xy_1x_1)\bigr]}_{=1-\xi_1}\\
&\stackrel{\textrm{Lemma \ref{lem:cone-probabilities}}}{=}& \underbrace{\sum_{\substack{l\in\N,\\ x\in \mathcal{W}_{k-1}}} \P\left[\begin{array}{c} X_{l-1}\notin C(x),\\ X_l=x\end{array}\right] \cdot (1-\xi_1)}_{=\P[\TT_{k-1}<\infty]=1}
\cdot \sum_{\substack{x_1\in V_1^\times,\\ y_1\in V_2^\times}} \P\left[\begin{array}{c} X_{m-1}\notin C(y_1x_1),\\ X_m=y_1x_1,\\ \forall m'<m: X_{m'}\notin V_1^\times\end{array}\right]\\
&=&\underbrace{\sum_{\substack{l\in\N,\\ x\in \mathcal{W}_0}} \P\left[\begin{array}{c} X_{l-1}\notin C(x),\\ X_l=x\end{array}\right] \cdot (1-\xi_1)}_{=\P[\TT_0<\infty]=1}
\cdot \sum_{\substack{x_1\in V_1^\times,\\ y_1\in V_2^\times}} \P\left[\begin{array}{c} X_{m-1}\notin C(y_1x_1),\\ X_m=y_1x_1,\\ \forall m'<m: X_{m'}\notin V_1^\times\end{array}\right]\\
&\stackrel{\textrm{Lemma \ref{lem:cone-probabilities}}}{=}& \sum_{\substack{l\in\N,\\ x\in \mathcal{W}_0}} \P\left[\begin{array}{c} X_{l-1}\notin C(x),\\ X_l=x\end{array}\right] 
\cdot \sum_{\substack{x_1\in V_1^\times,\\ y_1\in V_2^\times}} \P_x\left[\begin{array}{c} X_{m-1}\notin C(xy_1x_1),\\ X_m=xy_1x_1,\\ \forall m'<m: X_{m'}\in C(x) \end{array}\right]\cdot (1-\xi_1)\\
&=& 
\P[\TT_1-\TT_0=m].
\end{eqnarray*}
This shows that the sequence $\bigl(\TT_{k+1}-\TT_k)_{k\in\N}$ is identically distributed.
\par
For the proof of independence, 
let be $k\in\N$ and \mbox{$m_1,\ldots,m_k\in\N$.} Then, as we have seen above, we have
\begin{eqnarray*}
&&\P\bigl[\TT_{j}-\TT_{j-1}=m_j\bigr] \\[1ex]
&=&\sum_{\substack{l\in\N,\\ w_{j-1}\in \mathcal{W}_{j-1}}} \sum_{\substack{x_j\in V_1^\times,\\ y_j\in V_2^\times}} \P\left[\begin{array}{c} X_{\TT_{j-1}}=w_{j-1},X_{\TT_{j}}=w_{j-1}y_jx_j,\\\TT_{j-1}=l,\TT_j=l+m_j\end{array}\right]\\
&=&\sum_{\substack{l\in\N,\\ w_{j-1}\in \mathcal{W}_{j-1},\\ x_j\in V_1^\times,\\ y_j\in V_2^\times}} \P\left[\begin{array}{c} X_{l-1}\notin C(w_{j-1}),\\ X_l=w_{j-1}\end{array}\right] \\
&&\quad 
\cdot  \P_{w_{j-1}}\left[\begin{array}{c} X_{m_j-1}\notin C(w_{j-1}y_jx_j), \\ X_{m_j}=w_{j-1}y_jx_j,\\ \forall m'<m: X_{m'}\in C(w_{j-1}) \end{array}\right]\cdot (1-\xi_1).
\end{eqnarray*}
In the following we set $w_0:=o$, and for $x_1,\ldots,x_k\in V_1^\times$, $y_1,\ldots,y_k\in V_2^\times$,  we set $w_j:=y_1x_1\ldots y_jx_j$ for $j\in\{1,\ldots,k\}$. Then we obtain:
\begin{eqnarray*}
&& \P\Bigl[ \TT_1-\TT_0=m_1, \TT_2-\TT_1=m_2,\ldots,\TT_k-\TT_{k-1} = m_k\Bigr] \\
&=& \sum_{\substack{x\in\mathcal{W}_0,\\ x_1,\ldots, x_k\in V_1^\times,\\ y_1,\ldots,y_k\in V_2^\times}}\P\left[[X_{\TT_0}=x]\cap \bigcap_{j=1}^k \left[\begin{array}{c} \TT_j-\TT_{j-1}=m_j, \\  X_{\TT_j}=xy_1x_1\ldots y_jx_j \end{array}\right]\right] \\
&=& \sum_{\substack{l\in\N,\\ x\in \mathcal{W}_0}} \P\left[\begin{array}{c} X_{l-1}\notin C(x),\\ X_l=x\end{array}\right] \\
&&\quad \cdot \sum_{\substack{x_1,\ldots, x_k\in V_1^\times,\\ y_1,\ldots,y_k\in V_2^\times}} \Biggl(\prod_{j=1}^k \P_{xw_{j-1}}\left[\begin{array}{c} X_{m_j-1}\notin C(xw_j), \\ X_{m_j}=xw_j,\\ \forall m'<m_j: X_{m'}\in C(xw_{j-1}) \end{array}\right]\Biggr)\cdot (1-\xi_1)
\end{eqnarray*}
\begin{eqnarray*}
&\stackrel{\textrm{Lemma \ref{lem:cone-probabilities}}}{=}& \sum_{\substack{l\in\N,\\ x\in \mathcal{W}_0}} \P\left[\begin{array}{c} X_{l-1}\notin C(x),\\ X_l=x\end{array}\right] \cdot
\sum_{\substack{x_1\in V_1^\times,\\ y_1\in V_2^\times}} \P_x\left[\begin{array}{c} X_{m_1-1}\notin C(xy_1x_1), \\ X_{m_1}=xy_1x_1,\\ \forall m'<m_1: X_{m'}\in C(x) \end{array}\right]\\
&&\quad \cdot \sum_{\substack{x_2,\ldots, x_k\in V_1^\times,\\ y_2,\ldots,y_k\in V_2^\times}} \Biggl( \prod_{j=2}^k 
\underbrace{\sum_{\substack{t\in\N,\\ w\in \mathcal{W}_{j-1}}} \P\left[\begin{array}{c} X_{t-1}\notin C(w),\\ X_t=w\end{array}\right] \cdot (1-\xi_1)}_{=\P[\TT_{j-1}<\infty]=1}\\
&& \quad \quad \cdot \P\left[\begin{array}{c} X_{m_j-1}\notin C(y_jx_j), \\ X_{m_j}=y_jx_j,\\ \forall m'<m_j: X_{m'}\notin V_1^\times   \end{array}\right]\Biggr) \cdot (1-\xi_1)\\
&\stackrel{\textrm{Lemma \ref{lem:cone-probabilities}}}{=}& \sum_{\substack{l\in\N,\\ x\in \mathcal{W}_0}} \P\left[\begin{array}{c} X_{l-1}\notin C(x),\\ X_l=x\end{array}\right] \cdot
\sum_{\substack{x_1\in V_1^\times,\\ y_1\in V_2^\times}} \P_x\left[\begin{array}{c} X_{m_1-1}\notin C(xy_1x_1), \\ X_{m_1}=xy_1x_1,\\ \forall m'<m_1: X_{m'}\in C(x) \end{array}\right]\cdot (1-\xi_1)\\
&&\quad \cdot \prod_{j=2}^k \sum_{x_j\in V_1^\times, y_j\in V_2^\times}  
\sum_{\substack{t\in\N,\\ w\in \mathcal{W}_{j-1}}} \P\left[\begin{array}{c} X_{t-1}\notin C(w),\\ X_t=w\end{array}\right] \\
&&\quad\quad \cdot \P_w\left[\begin{array}{c} X_{m_j-1}\notin C(wy_jx_j), \\ X_{m_j}=wy_jx_j,\\ \forall m'<m_j: X_{m'}\in C(w) \end{array}\right] \cdot (1-\xi_1)\\
&=&\prod_{j=1}^k \P[\TT_j-\TT_{j-1}=m_j].
\end{eqnarray*}
Hence, we have shown  independence of the sequence $(\TT_k-\TT_{k-1})_{k\in\N}$. This finishes the proof.
\end{proof}

Finally, for random walks on $V$ depending on a finite number $d\in\N$ of parameters $p_1,\ldots,p_d\in(0,1)$, we show that $\E[\TT_1-\TT_0]$ varies real-analytically when seen as a function in $(p_1,\dots,p_d)\in\mathcal{P}_d$.  

\begin{Prop}\label{prop:T1-T0-analytic}
Assume that the random walk on $V$ depends on finitely many parameters $p_1,\ldots,p_d\in(0,1)$, $d\in\N$. Then the mapping
$$
\mathcal{P}_d\ni (p_1,\ldots,p_d)\mapsto \E[\TT_1-\TT_0]
$$
is real-analytic.
\end{Prop}
\begin{proof}
For $n\in\N$, we can rewrite:
\begin{eqnarray*}
&&\P[\TT_1-\TT_0=n] \\[1ex]
&=& \sum_{\substack{m\in\N,\\ x\in \mathcal{W}_0,\\ x_1\in V_1^\times, y_1\in V_2^\times}} \P\Bigl[X_{\TT_0}=x,X_{\TT_1}=xy_1x_1,\TT_0=m,\TT_1=m+n\Bigr]
\end{eqnarray*}
\begin{eqnarray}
&=&  \sum_{\substack{m\in\N,\\ x\in \mathcal{W}_0}} \P\left[ \begin{array}{c} X_{m-1}\notin C(x),\\ X_m=x\end{array}\right]\cdot  \sum_{\substack{x_1\in V_1^\times,\\ y_1\in V_2^\times}} \P_x\left[\begin{array}{c} X_{n-1}\notin C(xy_1x_1),\\ X_n=xy_1x_1,\\ \forall t<n: X_t\in C(x)\end{array}\right]\nonumber \\
&&\quad \cdot \underbrace{\P_{xy_1x_1}\bigl[\forall j\geq 1: X_j\in C(xy_1x_1)\bigr]}_{=1-\xi_1} \nonumber\\
&\stackrel{\textrm{Lemma \ref{lem:cone-probabilities}}}{=}& \underbrace{\sum_{\substack{m\in\N,\\ x\in \mathcal{W}_0}} \P\left[ \begin{array}{c} X_{m-1}\notin C(x),\\ X_m=x\end{array}\right]\cdot  (1-\xi_1)}_{=\P[\TT_0<\infty]=1} \cdot \sum_{\substack{x_1\in V_1^\times,\\ y_1\in V_2^\times}} \P\left[\begin{array}{c} X_{n-1}\notin C(y_1x_1),\\ X_n=y_1x_1,\\ \forall t<n: X_t\notin V_1^\times\end{array}\right] \nonumber\\
&=& \sum_{x_1\in V_1^\times, y_1\in V_2^\times} \P\left[\begin{array}{c} X_{n-1}\notin C(y_1x_1),\\ X_n=y_1x_1,\\ \forall t<n: X_t\notin V_1^\times\end{array}\right]. \label{equ:T1-T0=n-prob}
\end{eqnarray}
The summands in the last sum describe probabilities which depend on paths of length $n\in\N$ only.
Therefore, the probabilities $\P[\TT_1-\TT_0=n]$ can be written in the form
\begin{equation}\label{equ:n-form}
\sum_{\substack{n_1,\ldots,n_d\geq 0:\\ n_1+\ldots + n_d=n }} c(n_1,\ldots,n_d)\cdot p_1^{n_1}\cdot \ldots \cdot p_d^{n_d},
\end{equation}
where $c(n_1,\ldots,n_d)\in \N_0$. Of course, we are only allowed  to vary the parameters  \mbox{$p_1,\ldots,p_d>0$} in such a way that these parameter values still allow a well-defined random walk on $V$.
Since $\sum_{n\geq 1} \P[\TT_1-\TT_0=n]\cdot z^n$ has radius of convergence strictly bigger than $1$ due to existence of exponential moments of $\TT_1-\TT_0$, we have for sufficiently small $\delta>0$ that 
\begin{eqnarray}
&& \sum_{n\geq 1} \P[\TT_1-\TT_0=n]\cdot (1+\delta)^n \label{equ:T1-T0-sum-finite}\\
&=& \sum_{n\geq 1} \sum_{\substack{n_1,\ldots,n_d\geq 0:\\ n_1+\ldots + n_d=n }} c(n_1,\ldots,n_d)\cdot \bigl(p_1(1+\delta)\bigr)^{n_1}\cdot \ldots \cdot \bigl(p_d(1+\delta)\bigr)^{n_d}<\infty.\nonumber
\end{eqnarray}
Therefore,
\begin{equation}
\frac{\partial}{\partial z}\biggl[ \sum_{n\geq 1} \P[\TT_1-\TT_0=n]\cdot z^n\biggr]\Biggl|_{z=1+\frac{\delta}{2}}<\infty, \label{equ:T1-T0-sum-finite-2}
\end{equation}
which implies that $\E[\TT_1-\TT_0]$ varies real-analytically in $(p_1,\ldots,p_d)$. For further detailed explanations on real-analyticity in terms of $p_1,\ldots,p_d$, we refer to the beginning of Section 5 in \cite[p. 299]{gilch:22}.
\end{proof}

\section{Central Limit Theorem for the Drift}
\label{sec:CLT}

In this section we will prove Theorem \ref{th:CLT}. For this purpose, we define for $k\in\N$
$$
\DD_k := d(X_{\TT_{k-1}},X_{\TT_{k}}) 
$$
and set
$$
\widetilde \DD_k:=  \DD_k -(\TT_k-\TT_{k-1})\cdot\lambda = d(X_{\TT_{k-1}},X_{\TT_k}) -(\TT_k-\TT_{k-1}) \cdot \lambda.
$$
Due to the structure of free products, any path from $o$ to $X_{\TT_k}$ has to pass through $X_{\TT_0}, X_{\TT_1}, \ldots, X_{\TT_{k-1}}$. Therefore, we have
\begin{equation}\label{equ:d-decomposition}
d(o,X_{\TT_k})=d(o,X_{\TT_0})+\sum_{j=1}^k d(X_{\TT_{j-1}},X_{\TT_j}) = d(o,X_{\TT_0})+\sum_{j=1}^k \DD_j.
\end{equation}

\begin{Prop}\label{prop:Dk-iid}
We have:
\begin{enumerate}[label=(\roman*)]
\item $\bigl(\DD_k\bigr)_{k\in\N}$ is an i.i.d. sequence.
\item $\bigl(\widetilde \DD_k\bigr)_{k\in\N}$ is an i.i.d. sequence.
\item If $P$ depends on finitely many parameters $p_1,\ldots,p_d$, then $\E[\DD_1]$ and $\E\bigl[  d(X_{\TT_{0}},X_{\TT_1}) (\TT_1-\TT_{0})\bigr]$ vary real-analytically in $(p_1,\ldots,p_d)\in\mathcal{P}_d$.
\end{enumerate}
\end{Prop}
\begin{proof}
Since the proofs of (i) and (ii) are completely analogous to the proof of Proposition \ref{prop:Tk-Tk-1-iid}, we outsource both proofs to Appendix \ref{app:proof-Dk-iid}.
\par
For the proof of (iii), we note that, for $n,N\in\N$, we can rewrite analogously to the calculations (\ref{equ:T1-T0=n-prob}) in the proof of Proposition \ref{prop:T1-T0-analytic}
$$
\P\Bigl[d(X_{\TT_0},X_{\TT_1})=N,\TT_1-\TT_0=n\Bigr] = \sum_{\substack{x_1\in V_1^\times,\\ y_1\in V_2^\times:\\ d(o,y_1x_1)=N}} \P\left[\begin{array}{c} X_{n-1}\notin C(y_1x_1),\\ X_n=y_1x_1,\\ \forall t<n: X_t\notin V_1^\times\end{array}\right].
$$
The summands in the right sum depend only on paths of length $n\in\N$. Therefore, we can rewrite $\P\bigl[d(X_{\TT_0},X_{\TT_1})=N,\TT_1-\TT_0=n\bigr]$ as a sum
\begin{equation}\label{equ:form-n-path}
\sum_{\substack{n_1,\ldots,n_d\geq 0:\\ n_1+\ldots + n_d=n }} c_N(n_1,\ldots,n_d)\cdot p_1^{n_1}\cdot \ldots \cdot p_d^{n_d},
\end{equation}
where $c_N(n_1,\ldots,n_d)\in\N$. For sufficiently small $\delta>0$, we have then
$$
\sum_{n,N\in\N} \P\left[ \begin{array}{c} d(X_{\TT_0},X_{\TT_1})=N,\\ \TT_1-\TT_0=n\end{array}\right]\cdot (1+\delta)^n = \sum_{n\in\N} \P\bigl[\TT_1-\TT_0=n\bigr]\cdot (1+\delta)^n \stackrel{(\ref{equ:T1-T0-sum-finite})}{<}\infty.
$$
We also note that
\begin{equation}\label{equ:zero-prob}
\P\Bigl[d(X_{\TT_0},X_{\TT_1})=N,\TT_1-\TT_0=n\Bigr]=0 \quad \textrm{ if } N>n.
\end{equation}
Since we can rewrite
$$
\E[\DD_1]=\E\bigl[d(X_{\TT_0},X_{\TT_1})\bigr]=\sum_{n,N\in\N} N\cdot \sum_{\substack{n_1,\ldots,n_d\geq 0:\\ n_1+\ldots + n_d=n }} c_N(n_1,\ldots,n_d)\cdot p_1^{n_1}\cdot \ldots \cdot p_d^{n_d}
$$
and
\begin{eqnarray*}
&& \sum_{n,N\in\N} N\cdot \sum_{\substack{n_1,\ldots,n_d\geq 0:\\ n_1+\ldots + n_d=n }} c_N(n_1,\ldots,n_d)\cdot \bigl(p_1(1+\delta/2)\bigr)^{n_1}\cdot \ldots \cdot \bigl(p_d(1+\delta/2)\bigr)^{n_d}\\
&=& \sum_{n,N\in\N} N\cdot \P\bigl[d(X_{\TT_0},X_{\TT_1})=N,\TT_1-\TT_0=n\bigr]\cdot (1+\delta/2)^n \\
&\stackrel{(\ref{equ:zero-prob})}{\leq}& \sum_{n,N\in\N} n\cdot \P\bigl[d(X_{\TT_0},X_{\TT_1})=N,\TT_1-\TT_0=n\bigr]\cdot (1+\delta/2)^n \\
&\leq & (1+\delta/2)\cdot \frac{\partial}{\partial z}\biggl[ \sum_{n\in\N} \P\bigl[\TT_1-\TT_0=n\bigr]\cdot z^n\biggr]\Biggl|_{z=1+\delta/2}\stackrel{(\ref{equ:T1-T0-sum-finite-2})}{<}\infty,
\end{eqnarray*}
we obtain that $\E[\DD_1]=\E\bigl[d(X_{\TT_0},X_{\TT_1})\bigr]$ varies real-analytically in the random walk parameters $(p_1,\ldots,p_d)\in\mathcal{P}_d$.
\par
Real-analyticity of $\E\bigl[  d(X_{\TT_{0}},X_{\TT_1}) (\TT_1-\TT_{0})\bigr]$ follows  analogously from 
\begin{eqnarray*}
&& \sum_{n,N\in\N} N\cdot n\cdot \P\bigl[d(X_{\TT_0},X_{\TT_1})=N,\TT_1-\TT_0=n\bigr]\cdot (1+\delta/2)^n \\
&\stackrel{(\ref{equ:zero-prob})}{\leq} & \sum_{n\in\N} n^2 \cdot \P\bigl[\TT_1-\TT_0=n\bigr]\cdot (1+\delta/2)^n <\infty,
\end{eqnarray*}
together with exponential moments of $\TT_1-\TT_0$. 
\end{proof}

For $n\in\N$, set
$$
\mathbf{k}(n):=\sup\bigl \{m\in\N \,\bigl|\, \TT_m\leq n\bigr\}.
$$
Then $\lim_{n\to\infty} \mathbf{k}(n)=\infty$ almost surely since $\TT_k<\infty$ almost surely for each $k\in\N_0$. Furthermore, due to Propositions \ref{prop:T1-T0-exp-mom} and \ref{prop:Tk-Tk-1-iid}, the Strong Law of Large Numbers and almost sure finiteness of $\TT_0$  give
\begin{equation}\label{equ:Tn/n-convergence}
\lim_{n\to\infty} \frac{\TT_{\mathbf{k}(n)}}{\mathbf{k}(n)} = \lim_{n\to\infty} \frac{1}{\mathbf{k}(n)}\sum_{k=1}^{\mathbf{k}(n)} (\TT_k-\TT_{k-1})=\E[\TT_1-\TT_0]\quad \textrm{almost surely.}
\end{equation}

We have the following alternative formula for the drift $\lambda$: 
\begin{Cor}\label{cor:lambda-formula}
$$
\lambda= \lim_{n\to\infty} \frac{d(o,X_n)}{n} = \frac{\E[\DD_1]}{\E[\TT_1-\TT_0]} \quad \textrm{almost surely.}
$$
\end{Cor}
\begin{proof}
Since $\lambda$ exists due to \cite[Corollary 4.2]{gilch} and $\TT_0<\infty$ almost surely, (\ref{equ:d-decomposition}) yields
$$
\lambda= \lim_{n\to\infty} \frac{d(o,X_{\TT_{\mathbf{k}(n)}})}{\TT_{\mathbf{k}(n)}} =  \lim_{k\to\infty} \frac{\mathbf{k}(n)}{\TT_{\mathbf{k}(n)}} \frac{1}{\mathbf{k}(n)} \sum_{i=1}^{\mathbf{k}(n)} \DD_i \quad \textrm{almost surely.}
$$
As $0\leq \DD_1\leq \TT_1-\TT_0$ we have $\E[\DD_1]<\infty$.
The Strong Law of Large Numbers together with Proposition \ref{prop:Dk-iid} and (\ref{equ:Tn/n-convergence}) yield now the proposed formula for the drift $\lambda$.
\end{proof}
We immediately obtain:
\begin{Cor}\label{cor:lambda-analytic}
If $P$ depends on finitely many parameters $p_1,\ldots, p_d$, then the mapping 
$$
\mathcal{P}_d\ni (p_1,\ldots,p_d) \mapsto \lambda=\lambda(p_1,\ldots,p_d)
$$
is real-analytic.
\end{Cor}
\begin{proof}
This follows directly from the formula for $\lambda$ in Corollary \ref{cor:lambda-formula} and Propositions \ref{prop:T1-T0-analytic} and  \ref{prop:Dk-iid}.(iii).
\end{proof}

Now we can show in the next two lemmas that the random variables $\widetilde\DD_k$ are centralized and have finite, non-trivial variance.

\begin{Lemma}
$\bar\sigma_{\lambda}^2:=\mathrm{Var}\bigl(\widetilde \DD_1\bigr)<\infty$.
\end{Lemma}
\begin{proof}
This follows from 
$$
0\leq \bigl| \DD_1 -(\TT_1-\TT_0)\cdot \lambda\bigr| \leq d(X_{\TT_0},X_{\TT_1}) +(\TT_1-\TT_0)\cdot \lambda \stackrel{\lambda\in(0,1]}{\leq} 2\cdot (\TT_1-\TT_0)
$$
and  the fact that $\TT_1-\TT_0$ has exponential moments \mbox{(Proposition \ref{prop:T1-T0-exp-mom}).}
\end{proof}

\begin{Lemma}\label{lem:variance-lambda>0}
We have $\E\bigl[\widetilde\DD_1\bigr]=0$ and 
$$
\bar\sigma_{\lambda}^2=\E\Bigl[ \Bigl( d(X_{\TT_{0}},X_{\TT_1}) -(\TT_1-\TT_{0}) \cdot \lambda\Bigr)^2\Bigr] >0.
$$
\end{Lemma}
\begin{proof}
Corollary \ref{cor:lambda-formula} implies
$$
\E\bigl[\widetilde \DD_1\bigr] = \E\Bigl[\DD_1-(\TT_1-\TT_0)\cdot\lambda\Bigr] = \E\bigl[\DD_1\bigr] - \lambda \cdot \underbrace{\E[\TT_1-\TT_0]}_{=\E[\DD_1]/\lambda}=0.
$$
The variance formula for $\bar\sigma_\lambda^2$ follows now directly from $\E\bigl[\widetilde \DD_1\bigr]=0$. Moreover, the difference $d(X_{\TT_{0}},X_{\TT_1}) -(\TT_1-\TT_{0}) \cdot \lambda$ is not constant, which we show by constructing two different paths yielding two different values for that difference with positive probability. Recall the assumption made at the beginning of Subsection \ref{subsec:random-walks}: w.l.o.g. there is $x_0\in V_1^\times$ with $p_1^{(m_0)}(o_1,x_0)>0$ and $p_1^{(n_0)}(x_0,x_0)>0$ for some $m_0,n_0\in\N$. Moreover, take any $y\in V_2^\times$ with $p_2(o_2,y)>0$. Then:
\begin{eqnarray*}
&&\P\Bigl[X_{\TT_0} =x_0,X_{\TT_1} =x_0yx_0, \TT_1-\TT_0=m_0+1\Bigr]\\
&\geq & \alpha_1^{m_0}p_1^{(m_0)}(o_1,x_0)\cdot \alpha_2\cdot p_2(o_2,y) \cdot  \alpha_1^{m_0}p_1^{(m_0)}(o_1,x_0)\cdot (1-\xi_1)>0.
\end{eqnarray*}
But we also have
\begin{eqnarray*}
&&\P\Bigl[X_{\TT_0} =x_0,X_{\TT_1-n_0} =X_{\TT_1}=x_0yx_0, \TT_1-\TT_0=m_0+1+n_0\Bigr]\\
&\geq & \alpha_1^{2m_0+n_0}p_1^{(m_0)}(o_1,x_0)^2 \cdot \alpha_2\cdot p_2(o_2,y) \cdot p_1^{(n_0)}(x_0,x_0)\cdot (1-\xi_1)>0,
\end{eqnarray*}
since we can add a loop of length $n_0$ at $x_0yx_0$ within $C(x_0y)\setminus \bigl(C(x_0yx_0)\setminus \{x_0yx_0\}\bigr)$. As $\lambda >0$ and we have in both cases $d(X_{\TT_0},X_{\TT_1})=d(x_0,x_0yx_0)$ but different values of $\TT_1-\TT_0$, we have shown that the difference $d(X_{\TT_0},X_{\TT_1})-(\TT_1-\TT_0)\lambda$ is not constant, providing a strictly positive variance $\bar\sigma_\lambda^2>0$.
\end{proof}

For $k\in\N$, define
$$
\SS_k := \sum_{j=1}^k \DD_j = \sum_{j=1}^k d(X_{\TT_{j-1}},X_{\TT_j}) = d(X_{\TT_0},X_{\TT_k})= d(o,X_{\TT_k})-d(o,X_{\TT_0}),
$$
and set
$$
\widetilde \SS_k :=\sum_{j=1}^k \widetilde \DD_j =\SS_k -(\TT_k-\TT_0)\cdot \lambda.
$$
After all this cumbersome preliminary work we can now follow the approach in \mbox{\cite[Section 4]{gilch:22}} in order to derive the proposed central limit theorem.

\begin{Lemma}\label{lem:overhang}
$$
\frac{d(o,X_n)-\SS_{\mathbf{k}(n)}}{\sqrt{n}} \xrightarrow{\P} 0.
$$
\end{Lemma}
\begin{proof}
Let be $\varepsilon >0$.
For $n\in\N$, we have
\begin{eqnarray*}
0& \leq & d(o,X_n)-\SS_{\mathbf{k}(n)} = d(X_{\TT_{\mathbf{k}(n)}},X_n) + d(o,X_{\TT_0}) \\
&\leq& n-\TT_{\mathbf{k}(n)} + \TT_0 \leq \TT_{\mathbf{k}(n)+1}-\TT_{\mathbf{k}(n)} +\TT_0. 
\end{eqnarray*}
Since both $\TT_0$ and $\TT_1-\TT_0$ have exponential moments,  we get:
\begin{eqnarray*}
&&\P\Bigl[ d(o,X_n)-\SS_{\mathbf{k}(n)} > \varepsilon \sqrt{n} , \mathbf{k}(n)\geq 1\Bigr]\\
&\leq & \P\Bigl[ \TT_{\mathbf{k}(n)+1}-\TT_{\mathbf{k}(n)}+\TT_0 > \varepsilon \sqrt{n} , \mathbf{k}(n)\geq 1\Bigr]\\
&\leq &  \P\Bigl[ \exists k\in\{1,\ldots,n\}: \TT_{k+1}-\TT_{k}+\TT_0 > \varepsilon \sqrt{n} , \mathbf{k}(n)\geq 1\Bigr]\\
&\leq & \P\Bigl[ \exists k\in\{1,\ldots,n\}: \TT_{k+1}-\TT_{k} > \frac{\varepsilon}{2} \sqrt{n}\Bigr] + \P\Bigl[ \TT_0 > \frac{\varepsilon}{2} \sqrt{n}  \Bigr]\\
&\stackrel{\textrm{Proposition \ref{prop:Tk-Tk-1-iid}}}{\leq} & n\cdot  \P\Bigl[ \TT_{1}-\TT_{0} > \frac{\varepsilon}{2} \sqrt{n}\Bigr] + \P\Bigl[ \TT_0 > \frac{\varepsilon}{2} \sqrt{n}  \Bigr]\\
&\leq & n\cdot  \P\biggl[ (\TT_{1}-\TT_{0})^4 > \frac{\varepsilon^4}{2^4} n^2\biggr] + \P\Bigl[ \TT_0 > \frac{\varepsilon}{2} \sqrt{n}  \Bigr]\\
&\stackrel{\textrm{Markov Inequality}}{\leq} & n\cdot \frac{\mathbb{E}\Bigl[(\TT_1-\TT_0)^4\Bigr]}{\frac{\varepsilon^4}{2^4} n^2} + \frac{\E[\TT_0]}{\frac{\varepsilon}{2} \sqrt{n}} \xrightarrow{n\to\infty} 0.
\end{eqnarray*}
Since $\mathbf{k}(n)\to\infty$ almost surely, we have proven the proposed claim.
\end{proof}
By Billingsley \cite[Theorem 14.4]{billingsley:99}, we have the convergence in distribution
$$
\frac{\widetilde\SS_{\mathbf{k}(n)}}{\bar\sigma_\lambda \cdot \sqrt{\mathbf{k}(n)}} \xrightarrow{\mathcal{D}} N(0,1).
$$
The convergence of (\ref{equ:Tn/n-convergence}) yields 
$$
0\leq \frac{n-\TT_{\mathbf{k}(n)}}{\mathbf{k}(n)} \leq \frac{\TT_{\mathbf{k}(n)+1}-\TT_{\mathbf{k}(n)}}{\mathbf{k}(n)} \xrightarrow{n\to\infty} 0\quad \textrm{almost surely.}
$$
This in turn implies
$$
\lim_{n\to\infty} \frac{n}{\mathbf{k}(n)}= \lim_{n\to\infty}  \frac{n-\TT_{\mathbf{k}(n)}}{\mathbf{k}(n)} +\frac{\TT_{\mathbf{k}(n)}}{\mathbf{k}(n)}=\E[\TT_1-\TT_0] \quad \textrm{almost surely.}
$$
Therefore, an application of the Lemma of Slutsky gives
\begin{equation}\label{equ:drift-CLT-convergence}
\frac{\widetilde\SS_{\mathbf{k}(n)}}{ \sigma_\lambda\cdot \sqrt{n}}
= \frac{\widetilde\SS_{\mathbf{k}(n)}}{ \sigma_\lambda\cdot \sqrt{\mathbf{k}(n)}}\frac{\sqrt{\mathbf{k}(n)}}{\sqrt{n}}
 \xrightarrow{\mathcal{D}} N(0,1),
\end{equation}
where
\begin{equation}\label{equ:variance-lambda}
\sigma_\lambda^2=\frac{\E\Bigl[ \Bigl( d(X_{\TT_{0}},X_{\TT_1}) -(\TT_1-\TT_{0}) \cdot \lambda\Bigr)^2\Bigr]}{\E[\TT_1-\TT_0]}.
\end{equation}

It remains to control the part of $d(o,X_n)- n\cdot\lambda$, which is not ``covered'' by $\widetilde{\SS}_{\mathbf{k}(n)}$. To this end, we need the following proposition:

\begin{Prop}\label{prop:overhang}
For each $\varepsilon>0$,
$$
\lim_{n\to\infty} \P\Bigl[ \Bigl|\widetilde \SS_{\mathbf{k}(n)} - \bigl( d(o,X_n)-n\cdot \lambda\bigr) \Bigr|> \varepsilon \cdot\sqrt{n} \Bigr] =0.
$$
\end{Prop}
\begin{proof}
Let be $\varepsilon>0$. Recall that
$$
\widetilde \SS_{\mathbf{k}(n)}=\SS_{\mathbf{k}(n)} - (\TT_{\mathbf{k}(n)}-\TT_0)\cdot \lambda.
$$
Therefore, for $n\in\N$, we get
\begin{eqnarray*}
&& \P\Bigl[ \Bigl|\widetilde \SS_{\mathbf{k}(n)}- \bigl( d(o,X_n)-n\cdot \lambda\bigr) \Bigr| > \varepsilon\cdot \sqrt{n},\mathbf{k}(n)\geq 1 \Bigr] \\
&=&  \P\Bigl[ \Bigl| \SS_{\mathbf{k}(n)}- (\TT_{\mathbf{k}(n)}-\TT_0)\cdot \lambda - \bigl( d(o,X_n)-n\cdot \lambda\bigr) \Bigr| > \varepsilon\cdot \sqrt{n},\mathbf{k}(n)\geq 1 \Bigr] \\
&\leq & \P\Bigl[  d(o,X_n) - \SS_{\mathbf{k}(n)}  > \frac{\varepsilon}{2}\cdot \sqrt{n},\mathbf{k}(n)\geq 1 \Bigr] \\
&&\quad + \P\Bigl[ \lambda \cdot \bigl( n- (\TT_{\mathbf{k}(n)}-\TT_0) \bigr)  > \frac{\varepsilon}{2}\cdot \sqrt{n},\mathbf{k}(n)\geq 1 \Bigr].
\end{eqnarray*}
Since $\mathbf{k}(n)\to\infty$ almost surely  Lemma \ref{lem:overhang} yields
$$
\lim_{n\to\infty}  \P\Bigl[  d(o,X_n) - \SS_{\mathbf{k}(n)}  > \frac{\varepsilon}{2}\cdot \sqrt{n},\mathbf{k}(n)\geq 1 \Bigr] =0.
$$
On the other hand side, we have
\begin{eqnarray*}
&& \P\Bigl[ \lambda \cdot \bigl( n- (\TT_{\mathbf{k}(n)}-\TT_0) \bigr)  > \frac{\varepsilon}{2}\cdot \sqrt{n}, \mathbf{k}(n)\geq 1 \Bigr]\\
&\leq &  \P\Bigl[ \lambda \cdot \bigl( \TT_{\mathbf{k}(n)+1}- (\TT_{\mathbf{k}(n)}-\TT_0) \bigr)  > \frac{\varepsilon}{2}\cdot \sqrt{n}, \mathbf{k}(n)\geq 1 \Bigr]\\
&\leq &  \P\Bigl[ \exists k\in\{1,\ldots,n\}:  \lambda \cdot \bigl( \TT_{k+1}- (\TT_{k}-\TT_0) \bigr)  > \frac{\varepsilon}{2}\cdot \sqrt{n} \Bigr]\\
&\leq &  \P\Bigl[ \exists k\in\{1,\ldots,n\}:    \TT_{k+1}- \TT_{k}   > \frac{\varepsilon}{4\lambda}\cdot \sqrt{n} \Bigr] + \P\Bigl[ \TT_0  > \frac{\varepsilon}{4\lambda}\cdot \sqrt{n} \Bigr]\\
&\stackrel{\textrm{Proposition \ref{prop:Tk-Tk-1-iid}}}{\leq} & n\cdot \P\Bigl[     \TT_{1}- \TT_{0}   > \frac{\varepsilon}{4\lambda}\cdot \sqrt{n} \Bigr] + \P\Bigl[ \TT_0  > \frac{\varepsilon}{4\lambda}\cdot \sqrt{n} \Bigr]\\
&=&  n\cdot \P\biggl[     (\TT_{1}- \TT_{0})^4   > \frac{\varepsilon^4}{(4\lambda)^4}\cdot n^2 \biggr] + \P\Bigl[ \TT_0  > \frac{\varepsilon}{4\lambda}\cdot \sqrt{n} \Bigr]\\
&\stackrel{\textrm{Markov Inequality}}{\leq} &  n\cdot \frac{\E\bigl[(\TT_1-\TT_0)^4\bigr]}{\frac{\varepsilon^4}{(4\lambda)^4}\cdot n^2} + \frac{\E[\TT_0]}{\frac{\varepsilon}{4\lambda}\cdot \sqrt{n}} 
\xrightarrow{n\to\infty} 0,
\end{eqnarray*}
where the exponential moments of $\TT_0$ and $\TT_1-\TT_0$ guarantee that the occuring expectations are finite.
Since $\mathbf{k}(n)\to\infty$ almost surely we have finally shown that
$$
\P\Bigl[ \widetilde \SS_{\mathbf{k}(n)} - \bigl( d(o,X_n)-n\cdot \lambda\bigr) > \varepsilon \cdot\sqrt{n} \Bigr] \xrightarrow{n\to\infty} 0.
$$
\end{proof}
Now we can give the proof of Theorem  \ref{th:CLT}:
\begin{proof}[Proof of Theorem \ref{th:CLT}.]
Proposition \ref{prop:overhang} and the convergence in (\ref{equ:drift-CLT-convergence}) together with an application of the Lemma of Slutsky finally give the proposed central limit theorem:
$$
\frac{d(o,X_n)-n\cdot \lambda}{\sqrt{n}} = \underbrace{\frac{d(o,X_n)-\widetilde\SS_{\mathbf{k}(n)} - n\cdot \lambda}{\sqrt{n}}}_{\xrightarrow{\P} 0} + \underbrace{\frac{\widetilde\SS_{\mathbf{k}(n)}}{\sqrt{n}}}_{\xrightarrow{\mathcal{D}} N(0,\sigma_\lambda^2)}\xrightarrow{\mathcal{D}} N(0,\sigma_\lambda^2).
$$
It remains to show that $\sigma^2_\lambda$ varies real-analytically if the random walk on $V$ depends on finitely many parameters $p_1,\ldots,p_d$, $d\in\N$. By the formula in (\ref{equ:variance-lambda}) and real-analyticity of $\E[\TT_1-\TT_0]$ (see Proposition \ref{prop:T1-T0-analytic}), it remains to show that the mapping
$$
\mathcal{P}_d \ni (p_1,\ldots,p_d)\mapsto \E\Bigl[ \Bigl( d(X_{\TT_{0}},X_{\TT_1}) -(\TT_1-\TT_{0}) \cdot \lambda\Bigr)^2\Bigr]
$$
varies real-analytically. Since
\begin{eqnarray*}
&& \E\Bigl[ \Bigl( d(X_{\TT_{0}},X_{\TT_1}) -(\TT_1-\TT_{0}) \cdot \lambda\Bigr)^2\Bigr]\\
&=& \E\bigl[  d(X_{\TT_{0}},X_{\TT_1})^2\bigr] - 2\cdot \lambda\cdot  \E\bigl[  d(X_{\TT_{0}},X_{\TT_1}) (\TT_1-\TT_{0})\bigr]+ \lambda^2 \cdot \E\bigl[ (\TT_1-\TT_{0})^2 \bigr]
\end{eqnarray*}
and
$$
\lambda=\frac{\E\bigl[  d(X_{\TT_{0}},X_{\TT_1})\bigr]}{\E[\TT_1-\TT_0]}
$$
it suffices -- in view of Propositions \ref{prop:T1-T0-analytic} and \ref{prop:Dk-iid}.(iii) and Corollary \ref{cor:lambda-analytic} -- to study the mappings
\begin{eqnarray*}
\mathcal{P}_d \ni (p_1,\ldots,p_d) &\mapsto  & \E\bigl[  (\TT_1-\TT_0)^2\bigr] , \\ 
 \mathcal{P}_d \ni (p_1,\ldots,p_d) &\mapsto  &\E\bigl[  d(X_{\TT_{0}},X_{\TT_1})^2\bigr].
\end{eqnarray*}
Real-analyticity of the first  mapping  follows with the same reasoning as in  the proof of  Proposition \ref{prop:T1-T0-analytic}, since the probabilities $\P\bigl[\TT_1-\TT_0=n\bigr]$ can be written as in  (\ref{equ:n-form}) and
for sufficiently small $\delta>0$ we have
\begin{eqnarray*}
&& \sum_{n\geq 1} n^2\cdot P\bigl[\TT_1-\TT_0=n\bigr] \cdot (1+\delta)^n \\
&=&(1+\delta)^2 \cdot \frac{\partial^2}{\partial^2 z} \biggl[\sum_{n\geq 2} \P\bigl[\TT_1-\TT_0=n\bigr]\cdot z^{n} \biggr]\Biggl|_{z=1+\delta} \\
&&\quad +(1+\delta) \cdot \frac{\partial}{\partial z} \biggl[\sum_{n\geq 1} \P\bigl[\TT_1-\TT_0=n\bigr]\cdot z^{n} \biggr]\Biggl|_{z=1+\delta}<\infty.
\end{eqnarray*}
Analogously to  the proof of Proposition \ref{prop:Dk-iid}.(iii), one can show  that $\E\bigl[  d(X_{\TT_{0}},X_{\TT_1})^2\bigr]$ varies real-analytically: once again, the probabilities \mbox{$\P\bigl[d(X_{\TT_0},X_{\TT_1})^2=N,\TT_1-\TT_0=n\bigr]$} can be written as in  (\ref{equ:form-n-path}) and, for sufficiently small $\delta>0$,
\begin{eqnarray*}
&& \sum_{n,N\in\N}N^2\cdot \P\bigl[d(X_{\TT_0},X_{\TT_1})^2=N^2,\TT_1-\TT_0=n\bigr]\cdot (1+\delta)^n\\
&\leq & \sum_{n\in\N}n^2\cdot \P\bigl[\TT_1-\TT_0=n\bigr]\cdot (1+\delta)^n<\infty.
\end{eqnarray*}
This  finishes the proof of Theorem \ref{th:CLT}.
\end{proof}

\section{Central Limit Theorem w.r.t. the Block Length}
\label{sec:CLT-2}

The reasoning for the proof of the Central Limit Theorem for the rate of escape w.r.t. the block length, Theorem \ref{th:CLT-2}, is very similar to the proof of Theorem \ref{th:CLT}. We replace in Section \ref{sec:CLT} the graph metric $d(\cdot,\cdot)$ by the block length $\Vert \cdot\Vert$ and redefine for $k\in \N$
$$
\DD_k := \Vert X_{\TT_k}\Vert - \Vert X_{\TT_{k-1}}\Vert = (2k+\tau) -\bigl(2(k-1)+\tau\bigr)=2,
$$
and set
$$
\widetilde \DD_k:= \DD_k - (\TT_k-\TT_{k-1})\cdot\ell = 2- (\TT_k-\TT_{k-1})\cdot\ell.
$$
In particular, the sequence $\bigl(\widetilde\DD_k\bigr)_{k\in\N}$ is again i.i.d., which follows now directly from Proposition \ref{prop:Tk-Tk-1-iid}.
We have
$$
\Vert X_{\TT_k}\Vert = 2k+\tau = \Vert X_{\TT_0} \Vert + \sum_{j=1}^k \DD_k.
$$
The analogue to Corollary \ref{cor:lambda-formula} is then given by
\begin{eqnarray}
\ell &=&\lim_{n\to\infty} \frac{\Vert X_n\Vert}{n}= \lim_{n\to\infty}\frac{\Vert X_{\TT_{\mathbf{k}(n)}}\Vert}{\TT_{\mathbf{k}(n)}}= \lim_{n\to\infty}\underbrace{\frac{\Vert X_{\TT_{\mathbf{k}(n)}}\Vert}{\mathbf{k}(n)}}_{\to 2} \underbrace{\frac{\mathbf{k}(n)}{\TT_{\mathbf{k}(n)}}}_{\to 1/\E[\TT_1-\TT_0]}\nonumber\\
&=& \frac{2}{\E[\TT_1-\TT_0]} \quad \textrm{ almost surely.} \label{equ:formula-ell}
\end{eqnarray}
We immediately obtain:
\begin{Cor}\label{cor:ell-analytic}
If $P$ depends on finitely many parameters $p_1,\ldots, p_d$, then the mapping 
$$
\mathcal{P}_d\ni (p_1,\ldots,p_d) \mapsto \ell=\ell(p_1,\ldots,p_d)
$$
is real-analytic.
\end{Cor}
\begin{proof}
This follows directly from (\ref{equ:formula-ell}) and Proposition \ref{prop:T1-T0-analytic}.
\end{proof}
Moreover:
\begin{Lemma}
We have $\E\bigl[\widetilde\DD_1\bigr]=0$ and
$$
\bar\sigma_\ell^2 := \mathrm{Var}\bigl(\widetilde\DD_1\bigr)= \E\Bigl[\bigl(2-(\TT_1-\TT_0)\cdot\ell\bigr)^2\Bigr]\in (0,\infty).
$$
\end{Lemma}
\begin{proof}
We have
$$
\E\bigl[\widetilde\DD_1\bigr]= 2- \ell \cdot \E[\TT_1-\TT_0]\stackrel{(\ref{equ:formula-ell})}{=} 2- \frac{2}{\E[\TT_1-\TT_0]} \cdot \E[\TT_1-\TT_0]=0,
$$
from which the proposed formula for $\bar\sigma_\ell^2$ follows. Furthermore,
since $\TT_1-\TT_0$ has exponential moments, we have $\bar\sigma_\ell^2<\infty$. The same reasoning as in the proof of Lemma \ref{lem:variance-lambda>0} shows that $\bar\sigma_\ell^2>0$. 
\end{proof}
The analogous definitions of $\SS_k$ and $\widetilde \SS_k$ now become
$$
\SS_k:=\sum_{j=1}^k \DD_j = 2k
$$
and
$$
\widetilde \SS_k:=\sum_{j=1}^k \widetilde \DD_j =2k-(\TT_k-\TT_0)\cdot \ell.
$$
We have then for $n\in\N$:
\begin{eqnarray*}
0&\leq& \Vert X_n\Vert - \SS_{\mathbf{k}(n)} \leq \Vert X_n\Vert - \Vert X_{\TT_{\mathbf{k}(n)}}\Vert + \Vert X_{\TT_{0}}\Vert \\
&\leq & n- \TT_{\mathbf{k}(n)}+\TT_{0}  \leq \TT_{\mathbf{k}(n)+1}-\TT_{\mathbf{k}(n)} +\TT_0.
\end{eqnarray*}
Then we can show completely analogously to the proof of Lemma \ref{lem:overhang} that
\begin{equation}\label{equ:overhang-block-length}
\frac{\Vert X_n\Vert - \SS_{\mathbf{k}(n)}}{\sqrt{n}}\xrightarrow{\P} 0.
\end{equation}
Since $(\TT_k-\TT_{k-1})_{k\in \N}$ is i.i.d., we obtain analogously to the reasoning in Section \ref{sec:CLT}:
\begin{equation}\label{equ:convergence-block-length}
\frac{\widetilde\SS_{\mathbf{k}(n)}}{ \sigma_\ell \cdot \sqrt{n}} \xrightarrow{\mathcal{D}} N(0,1),
\end{equation}
where
\begin{equation}\label{equ:variance-ell}
\sigma_\ell^2= \frac{\E\Bigl[\bigl(2-(\TT_1-\TT_0)\cdot \ell\bigr)^2\Bigr]}{\E[\TT_1-\TT_0]}.
\end{equation}
It remains to prove the analogue of Proposition \ref{prop:overhang}:
\begin{Prop}\label{prop:overhang-block-length}
For every $\varepsilon >0$,
$$
\lim_{n\to\infty} \P\Bigl[ \Bigl|\widetilde \SS_{\mathbf{k}(n)} - \bigl( \Vert X_n\Vert-n\cdot \ell\bigr) \Bigr|>\varepsilon \cdot  \sqrt{n}\Bigr] = 0.
$$
\end{Prop}
\begin{proof}
Observe that
$$
\widetilde \SS_{\mathbf{k}(n)}=2\cdot \mathbf{k}(n)  - (\TT_{\mathbf{k}(n)}-\TT_0)\cdot \ell.
$$
Therefore, for each $n\in\N$ and every $\varepsilon>0$, we get
\begin{eqnarray*}
&& \P\Bigl[ \Bigl|\widetilde \SS_{\mathbf{k}(n)}- \bigl( \Vert X_n\Vert -n\cdot \ell\bigr) \Bigr| > \varepsilon\cdot \sqrt{n},\mathbf{k}(n)\geq 1 \Bigr] \\
&=&  \P\Bigl[ \Bigl| 2\cdot \mathbf{k}(n)   - (\TT_{\mathbf{k}(n)}-\TT_0)\cdot \ell- \bigl( \Vert X_n\Vert -n\cdot \ell\bigr) \Bigr| > \varepsilon\cdot \sqrt{n},\mathbf{k}(n)\geq 1 \Bigr] \\
&\leq & \P\Bigl[  \Vert X_n\Vert - 2\cdot \mathbf{k}(n)  > \frac{\varepsilon}{2}\cdot \sqrt{n},\mathbf{k}(n)\geq 1 \Bigr] \\
&&\quad + \P\Bigl[ \ell \cdot \bigl( n- (\TT_{\mathbf{k}(n)}-\TT_0) \bigr)  > \frac{\varepsilon}{2}\cdot \sqrt{n},\mathbf{k}(n)\geq 1 \Bigr].
\end{eqnarray*}
Once again, by   (\ref{equ:overhang-block-length}), 
$$
\frac{ \Vert X_n\Vert - 2\cdot \mathbf{k}(n)}{\sqrt{n}}=\frac{\Vert X_n\Vert -\SS_{\mathbf{k}(n)}}{\sqrt{n}} \xrightarrow{\P} 0.
$$
The rest follows as in the proof of Proposition \ref{prop:overhang}.
\end{proof}

\begin{proof}[Proof of Theorem \ref{th:CLT-2}.]
Proposition \ref{prop:overhang-block-length} and the convergence in (\ref{equ:convergence-block-length}) together with an application of the Lemma of Slutsky finally give the proposed central limit theorem:
$$
\frac{\Vert X_n\Vert -n\cdot \ell}{\sqrt{n}} = \underbrace{\frac{\Vert X_n\Vert -\widetilde\SS_{\mathbf{k}(n)} - n\cdot \ell}{\sqrt{n}}}_{\xrightarrow{\P} 0} + \underbrace{\frac{\widetilde\SS_{\mathbf{k}(n)}}{\sqrt{n}}}_{\xrightarrow{\mathcal{D}} N(0, \sigma_\ell^2)}\xrightarrow{\mathcal{D}} N(0, \sigma_\ell^2).
$$
If $P$ depends on finitely many parameters $p_1,\ldots,p_d\in(0,1)$ only, then real-analyticity of $\sigma_\ell^2$ follows directly from the formula (\ref{equ:variance-ell}); recall that we have shown in the the proof of \mbox{Theorem \ref{th:CLT}} that $\E\bigl[ (\TT_1-\TT_0)^2\bigr]$ varies real-analytically and that $\ell$ varies real-analytically by Corollary \ref{cor:ell-analytic}.
%
%
\end{proof}

\section{Central Limit Theorem for the Entropy}
\label{sec:CLT-3}

The reasoning for the proof of the Central Limit Theorem for the asymptotic entropy, Theorem \ref{th:CLT-3}, is also similar to the proof of Theorem \ref{th:CLT} but needs some additional effort. We replace in Section \ref{sec:CLT} the graph metric $d(\cdot,\cdot)$ by the \textit{distance function} 
$$
d_L(o,x):=-\log L(o,x|1),\quad x\in V.
$$ 
Moreover, we assume that $P$ is $\varepsilon_0$-uniform for some $\varepsilon_0>0$. By \cite[Theorem 3.7]{gilch:11}, the asymptotic entropy $h$ satisfies
$$
h=\lim_{n\to\infty} \frac{d_L(o,X_n)}{n}= \lim_{n\to\infty}-\frac1n \log L(o,X_n|1) \quad \textrm{almost surely.}
$$
If $X_{\TT_{k-1}}=x\in V$ and $X_{\TT_{k}}=xy_kx_k\in V$ for $k\in\N$, $x_k\in V_1^\times$ and $y_k\in V_2^\times$, then we set $\mathbf{W}_k:=y_kx_k$.
Due to the structure of the free product any path from $o$ to $X_n$ has to pass through $X_{\TT_1},\ldots, X_{\TT_{\mathbf{k}(n)}}$. By
(\ref{eq:L-L}), we can rewrite
$$
L(o,X_n|1) = L(o,X_{\TT_0}|1)\cdot \prod_{k=1}^{\mathbf{k}(n)} \underbrace{L(X_{\TT_{k-1}},X_{\TT_k}|1)}_{=L(o,\mathbf{W}_k|1)\textrm{ (Lemma \ref{lem:cone-probabilities})}}\cdot L(X_{\TT_{\mathbf{k}(n)}},X_n|1),
$$
which in turn yields
\begin{equation}\label{equ:dL-decomposition}
d_L(o,X_n) = d_L(o,X_{\TT_0})+ \sum_{k=1}^{\mathbf{k}(n)} d_L(o,\mathbf{W}_k|1) + d_L(X_{\TT_{\mathbf{k}(n)}},X_n).
\end{equation}
We remark that, for $x_1\in V_1^\times,y_1\in V_2^\times$, and  small $\delta\geq 0$ with $\xi_1(1+\delta)<1$ and $\xi_2(1+\delta)<1$, we have due to (\ref{lem:4-1equ-2})
\begin{eqnarray}
\log L(o,y_1x_1|1+\delta) &=& \log \Bigl( L_2\bigl(o_2,y_1\,\bigl|\,\xi_2(1+\delta)\bigr)\cdot L_1(o_1,x_1\,\bigl|\,\xi_1(1+\delta)\bigr)  \Bigr)\nonumber\\
&\leq& \log \frac{1}{\bigl(1-\xi_1(1+\delta)\bigr)\bigl(1-\xi_2(1+\delta)\bigr)}=:C_L(\delta)<\infty. \label{equ:L-bounded}
\end{eqnarray}
For $\delta=0$, we just write $C_L:=C_L(0)$. 
\par
On the other hand side, $L(o,y_1x_1|1)\geq \varepsilon_0^{d(o,y_1x_1)}$. Hence,
\begin{equation}\label{equ:logL-bound}
\bigl|-\log L(o,y_1x_1|1)\bigr| \leq \max\bigl\{ (-\log \varepsilon_0) \cdot d(o,y_1x_1), C_L\bigr\}.
\end{equation}
We redefine for $k\in \N$
$$
\DD_k :=-\log L(o,\mathbf{W}_k|1),
$$
and set
$$
\widetilde \DD_k:= \DD_k - (\TT_k-\TT_{k-1})\cdot h = -\log L(o,\mathbf{W}_k|1)- (\TT_k-\TT_{k-1})\cdot h.
$$

\begin{Prop}~ \label{prop:D-iid-entropy}
\begin{enumerate}[label=(\roman*)]
\item $(\DD_k)_{k\in\N}=\bigl(d_L(o,\mathbf{W}_k|1)\bigr)_{k\in\N}$ is an i.i.d. sequence.
\item $\bigl(\widetilde \DD_k\bigr)_{k\in\N}$ is an i.i.d. sequence.
\end{enumerate}
\end{Prop}
\begin{proof}
\begin{enumerate}[label=(\roman*)]
\item If $z\in\mathbb{R}$ with $\P[d_L(o,\mathbf{W}_k)=z]>0$, then we can show  
$$
\P\bigl[d_L(o,\mathbf{W}_k)=z\bigr]=\P\bigl[d_L(o,\mathbf{W}_1)=z\bigr]\quad \textrm{ for all $k\in\N$}
$$ 
completely analogously as in  the proof of Proposition \ref{prop:Dk-iid}.(i) by replacing the condition $d(o,y_1x_1)=n$ with  $d_L(o,y_1x_1)=z$.
\par
Similarily, if we have $k\in\N$ and $z_1,\ldots,z_k\in\mathbb{R}$ satisfying 
$$
\P\biggl[\bigcap_{j=1}^k \bigl[d_L(o,\mathbf{W}_j)=z_j\bigr]\biggr]>0,
$$ 
then we can show 
$$
\P\Biggl[\bigcap_{j=1}^k \bigl[d_L(o,\mathbf{W}_j)=z_j\bigr]\Biggr]=\prod_{j=1}^k \P\bigl[ d_L(o,\mathbf{W}_j)=z_j\bigr]
$$
completely analogously as in  the proof of Proposition \ref{prop:Dk-iid}.(i) by replacing the conditions $d(o,y_jx_j)=n_j$ with  $d_L(o,y_jx_j)=z_j$.
\item Once again, the proof works completely analogously as the  proof of Proposition \ref{prop:Dk-iid}.(ii) by replacing the conditions $d(o,y_jx_j)-n\cdot \lambda=z_j$ with  $d_L(o,y_jx_j)-n\cdot h=z_j$.
\end{enumerate}
\end{proof}

\begin{Lemma}\label{lem:D1-exp-moments-entropy}
$\DD_1$ has exponential moments.
\end{Lemma}
\begin{proof}
%
%
%
%
We have
\begin{eqnarray*}
0&\leq& \bigl|-\log L(o,\mathbf{W}_1|1)\bigr| \stackrel{(\ref{equ:logL-bound}) }{\leq} \max\bigl\{C_L,-\log  \varepsilon_0^{d(o,\mathbf{W}_1)}\bigr\} \\
&\leq& \max\bigl\{C_L,-\log  \varepsilon_0^{\TT_1-\TT_0}\bigr\}= \max\bigl\{C_L,(-\log\varepsilon_0) \cdot (\TT_1-\TT_0)\bigr\}.
\end{eqnarray*}
Since  
$\TT_1-\TT_0$ has exponential moments, the claim follows immediately.
\end{proof}

The last lemma guarantees that $\E[\DD_1]<\infty$, and together with Proposition \ref{prop:D-iid-entropy} and (\ref{equ:dL-decomposition}) we obtain then the analogue to Corollary \ref{cor:lambda-formula}:
\begin{eqnarray}
h &=&\lim_{n\to\infty} \frac{d_L(o,X_n)}{n}= \lim_{n\to\infty}\frac{d_L(o, X_{\TT_{\mathbf{k}(n)}})}{\TT_{\mathbf{k}(n)}}= \lim_{n\to\infty}\underbrace{\frac{d_L(o, X_{\TT_{\mathbf{k}(n)}})}{\mathbf{k}(n)}}_{\to \E[\DD_1]} \underbrace{\frac{\mathbf{k}(n)}{\TT_{\mathbf{k}(n)}}}_{\to 1/\E[\TT_1-\TT_0]}\nonumber\\
&=& \frac{\E[\DD_1]}{\E[\TT_1-\TT_0]} \quad \textrm{ almost surely.}\label{equ:entropy-formula-expectations}
\end{eqnarray}
We immediately obtain:
\begin{Cor}\label{cor:h-analytic}
If $P$ depends on finitely many parameters $p_1,\ldots, p_d$, then the mapping 
$$
\mathcal{P}_d\ni (p_1,\ldots,p_d) \mapsto h=h(p_1,\ldots,p_d)
$$
is real-analytic.
\end{Cor}
\begin{proof}
This follows directly from (\ref{equ:entropy-formula-expectations}) together with Proposition \ref{prop:T1-T0-analytic} and Lemma \ref{lem:D1-exp-moments-entropy}.
\end{proof}

Moreover:
\begin{Lemma}
We have $\E\bigl[\widetilde\DD_1\bigr]=0$ and
$$
\bar\sigma_h^2 := \mathrm{Var}\bigl(\widetilde\DD_1\bigr)= \E\Bigl[\Bigl(-\log L(o,\mathbf{W}_1|1)-(\TT_1-\TT_0)\cdot h\Bigr)^2\Bigr]\in (0,\infty).
$$
\end{Lemma}
\begin{proof}
We have
$$
\E\bigl[\widetilde\DD_1\bigr]= \E[\DD_1]- h \cdot \E[\TT_1-\TT_0]\stackrel{(\ref{equ:entropy-formula-expectations})}{=} \E[\DD_1]- \frac{\E[\DD_1]}{\E[\TT_1-\TT_0]} \cdot \E[\TT_1-\TT_0]=0,
$$
from which the proposed formula for $\bar\sigma_h^2$ follows. Furthermore,
since $\DD_1$ and $\TT_1-\TT_0$ have exponential moments, we have $\bar\sigma_h^2<\infty$. The same reasoning as in the proof of Lemma \ref{lem:variance-lambda>0} shows that $\bar\sigma_h^2>0$. 
\end{proof}

The analogous definitions of $\SS_k$ and $\widetilde \SS_k$ now become
$$
\SS_k:=\sum_{j=1}^k \DD_j = -\log L(o,X_{\TT_k}|1)-\log L(o,X_{\TT_0}|1)
$$
and
$$
\widetilde \SS_k:=\sum_{j=1}^k \widetilde \DD_j.
$$
We have then for $n\in\N$:
\begin{eqnarray*}
0&\leq & d_L(o, X_n) - \SS_{\mathbf{k}(n)} 
\leq -\log L(X_{\TT_{\mathbf{k}(n)}},X_n|1)-\log L(o,X_{\TT_0}|1)\\
&\leq& -\log \varepsilon_0^{n-\TT_{\mathbf{k}(n)}} -\log \varepsilon_0^{\TT_0} \leq (-\log\varepsilon_0)\cdot \bigl(\TT_{\mathbf{k}(n)+1}-\TT_{\mathbf{k}(n)}+\TT_0\bigr).
\end{eqnarray*}
Then we can show completely analogously to the proof of Lemma \ref{lem:overhang} that
\begin{equation}\label{equ:overhang-entropy}
\frac{d_L(o, X_n)  - \SS_{\mathbf{k}(n)}}{\sqrt{n}}\xrightarrow{\P} 0.
\end{equation}
Analogously to the reasoning in Section \ref{sec:CLT} we obtain
\begin{equation}\label{equ:convergence-entropy}
\frac{\widetilde\SS_{\mathbf{k}(n)}}{ \sigma_h \cdot \sqrt{n}} \xrightarrow{\mathcal{D}} N(0,1),
\end{equation}
where
\begin{equation}\label{equ:variance-h}
\sigma_h^2= \frac{\E\Bigl[\Bigl(-\log L(o,\mathbf{W}_1|1)-(\TT_1-\TT_0)\cdot h\Bigr)^2\Bigr]}{\E[\TT_1-\TT_0]}.
\end{equation}
It remains to prove the analogue of Proposition \ref{prop:overhang}:
\begin{Prop}\label{prop:overhang-entropy}
For every $\varepsilon >0$,
$$
\lim_{n\to\infty} \P\Bigl[ \Bigl|\widetilde \SS_{\mathbf{k}(n)} - \bigl( d_L(o, X_n) -n\cdot h\bigr) \Bigr|>\varepsilon \cdot  \sqrt{n}\Bigr] = 0.
$$
\end{Prop}
\begin{proof}
Observe that
$$
\widetilde \SS_{\mathbf{k}(n)}=\SS_{\mathbf{k}(n)}  - (\TT_{\mathbf{k}(n)}-\TT_0)\cdot h.
$$
Therefore, for each $n\in\N$ and every $\varepsilon>0$, we get
\begin{eqnarray*}
&& \P\Bigl[ \Bigl|\widetilde \SS_{\mathbf{k}(n)}- \bigl( d_L(o, X_n) -n\cdot h\bigr) \Bigr| > \varepsilon\cdot \sqrt{n},\mathbf{k}(n)\geq 1 \Bigr] \\
&=& \P\Bigl[ \Bigl|  \SS_{\mathbf{k}(n)}  - (\TT_{\mathbf{k}(n)}-\TT_0)\cdot h- \bigl( d_L(o, X_n)  -n\cdot h\bigr) \Bigr| > \varepsilon\cdot \sqrt{n},\mathbf{k}(n)\geq 1 \Bigr] \\
&\leq & \P\Bigl[  \bigl| d_L(o, X_n)  - \SS_{\mathbf{k}(n)} \bigr|  > \frac{\varepsilon}{2}\cdot \sqrt{n},\mathbf{k}(n)\geq 1 \Bigr] \\
&&\quad + \P\Bigl[ h \cdot \bigl( n- (\TT_{\mathbf{k}(n)}-\TT_0) \bigr)  > \frac{\varepsilon}{2}\cdot \sqrt{n},\mathbf{k}(n)\geq 1 \Bigr].
\end{eqnarray*}
Once again, by   (\ref{equ:overhang-entropy}), $\lim_{n\to\infty} P\Bigl[  \bigl| d_L(o, X_n)  - \SS_{\mathbf{k}(n)} \bigr|  > \frac{\varepsilon}{2}\cdot \sqrt{n},\mathbf{k}(n)\geq 1 \Bigr]=0$.
The rest follows as in the proof of Proposition \ref{prop:overhang}.
\end{proof}

\begin{proof}[Proof of Theorem \ref{th:CLT-3}.]
Proposition \ref{prop:overhang-entropy} and the convergence in (\ref{equ:convergence-entropy}) together with an application of the Lemma of Slutsky finally yield the proposed central limit theorem:
$$
\frac{d_L(o, X_n) -n\cdot h}{\sqrt{n}} = \underbrace{\frac{d_L(o, X_n) -\widetilde \SS_{\mathbf{k}(n)} - n\cdot h}{\sqrt{n}}}_{\xrightarrow{\P} 0} + \underbrace{\frac{\widetilde\SS_{\mathbf{k}(n)}}{\sqrt{n}}}_{\xrightarrow{\mathcal{D}} N(0, \sigma_h^2)}\xrightarrow{\mathcal{D}} N(0, \sigma_h^2).
$$
\end{proof}

It remains to prove Corollary \ref{cor:CLT-3}. From now on we consider free products of \textit{finite} graphs (that is, $V_1$ and $V_2$ are finite) and we assume that $P$ depends on finitely many parameters $p_1,\ldots,p_d$ only.
We have:
\begin{Lemma}\label{lem:logL-exp-mom}
$\E\bigl[-\log L(o,\mathbf{W}_1|1)\bigr]$ and $\E\bigl[-\log L(o,\mathbf{W}_1|1)(\TT_1-\TT_0)\bigr]$ vary real-analytically in $p_1,\ldots,p_d$.
\end{Lemma}
\begin{proof}
We follow the reasoning as in the proof of Proposition \ref{prop:Dk-iid}.(iii). For $n\in\N$ and $x_1\in V_1^\times$, $y_1\in V_1^\times$,
we can rewrite the following probability as in (\ref{equ:T1-T0=n-prob}) in the proof of Proposition \ref{prop:T1-T0-analytic}:
$$
\P\Bigl[\mathbf{W}_1=y_1x_1,\TT_1-\TT_0=n\Bigr] = \P\left[\begin{array}{c} X_{n-1}\notin C(y_1x_1),\\ X_n=y_1x_1,\\ \forall t<n: X_t\notin V_1^\times\end{array}\right].
$$
The probability on the right hand side depends only on paths of length $n\in\N$. Therefore, we can rewrite $\P\bigl[\mathbf{W}_1=y_1x_1,\TT_1-\TT_0=n\bigr]$ as a sum
\begin{equation}\label{equ:lem23-cn}
\sum_{\substack{n_1,\ldots,n_d\geq 0:\\ n_1+\ldots + n_d=n }} c_{y_1x_1,n}(n_1,\ldots,n_d)\cdot p_1^{n_1}\cdot \ldots \cdot p_d^{n_d},
\end{equation}
where $c_{y_1x_1,n}(n_1,\ldots,n_d)\in\N$. For $x_1\in V_1^\times, y_1\in V_2^\times$, $z>0$, set
$$
\mathcal{E}(y_1x_1|z):=\sum_{n\in\N} \P\left[ \begin{array}{c} \mathbf{W}_1=y_1x_1,\\ \TT_1-\TT_0=n\end{array}\right]\cdot z^n.
$$
For sufficiently small $\delta>0$, we have then
$$
\mathcal{E}(y_1x_1|1+\delta):=\sum_{n\in\N} \P\left[ \begin{array}{c} \mathbf{W}_1=y_1x_1,\\ \TT_1-\TT_0=n\end{array}\right]\cdot (1+\delta)^n \leq \sum_{n\in\N} \P\bigl[\TT_1-\TT_0=n\bigr]\cdot (1+\delta)^n \stackrel{\textrm{Prop.\ref{prop:T1-T0-exp-mom}}}{<}\infty,
$$
that is, $\mathcal{E}(y_1x_1|1)$ varies real-analytically in $(p_1,\ldots,p_d)\in\mathcal{P}_d$ due to (\ref{equ:lem23-cn}).
For a moment, fix  any \mbox{$\underline{p}=(p_1,\ldots,p_d)\in \mathcal{P}_d$.} For $x_1\in V_1^\times, y_1\in V_2^\times$, the function $L(o,y_1x_1|1)$ is obviously  real-analytic in $\underline{p}$, since $L(o,y_1x_1|z)$ has radius of convergence strictly bigger than $1$. Since $L(o,y_1x_1|1)$ is continuous in $\underline{p}$ and due to \mbox{$L(o,y_1x_1|1)\geq \varepsilon_0^N$,} where $N:=|V_1|+|V_2|$, there exists a complex neighbourhood $\mathcal{U}(\underline{p})$ of $\underline{p}$, where $L(o,y_1x_1|1)$ does not take values in $(-\infty,0]$. Due to finiteness of $V_1$ and $V_2$ this neighbourhood $\mathcal{U}(\underline{p})$ can be chosen independently of $x_1,y_1$ such that $-\log L(o,y_1x_1|1)$ varies real-analytically on $\mathcal{U}(\underline{p})$.
Then
\begin{eqnarray*}
\E[\DD_1] &= & \E\bigl[-\log L(o,\mathbf{W}_1|1)\bigr] =  \sum_{\substack{n\in\N,\\ x_1\in V_1^\times,y_1\in V_2^\times}}  -\log L(o,y_1x_1|1)\cdot \P\left[\begin{array}{c}\mathbf{W}_1=y_1x_1,\\ \TT_1-\TT_0=n\end{array}\right]\\
&=& \sum_{x_1\in V_1^\times,y_1\in V_2^\times}-\log L(o,y_1x_1|1)\cdot \mathcal{E}(y_1x_1|1),
\end{eqnarray*}
where the right hand side is a finite sum of products of real-analytic functions in $\underline{p}\in\mathcal{P}_d$. This proves real-analyticity of $\E[\DD_1]$.

For real-analyticity of $\E\bigl[  -\log L(o,\mathbf{W}_1|1) (\TT_1-\TT_{0})\bigr]$, observe that for sufficiently small $\delta>0$:
\begin{eqnarray*}
&& \sum_{\substack{n\in\N,\\ x_1\in V_1^\times, y_1\in V_2^\times}} \bigl|-\log L(o,y_1x_1|1)\bigr|\cdot n\cdot \P\left[\begin{array}{c}\mathbf{W}_1=y_1x_1,\\ \TT_1-\TT_0=n\end{array}\right]\cdot (1+\delta)^n \\
&\leq & \sum_{x_1\in V_1^\times, y_1\in V_2^\times} \bigl|-\log L(o,y_1x_1|1)\bigr|\cdot \sum_{n\in\N} n \cdot \P\bigl[\TT_1-\TT_0=n\bigr]\cdot (1+\delta)^n \stackrel{\textrm{Prop.\ref{prop:T1-T0-exp-mom}}}{<}\infty.
\end{eqnarray*}
Moreover, for $x_1\in V_1^\times, y_1\in V_2^\times$,
$$
\mathcal{E}'(y_1x_1|1+\delta)  = \sum_{n\in\N} n\cdot \P\left[\begin{array}{c}\mathbf{W}_1=y_1x_1,\\ \TT_1-\TT_0=n\end{array}\right] \cdot (1+\delta)^{n-1} <\infty.
$$
Since $\P\left[\mathbf{W}_1=y_1x_1,\TT_1-\TT_0=n\right]$ has the form  (\ref{equ:lem23-cn}) we obtain that $\mathcal{E}'(y_1x_1|1)$ varies real-analytically in $\underline{p}\in\mathcal{P}_d$.
Finally,
\begin{eqnarray*}
\E\bigl[  -\log L(o,\mathbf{W}_1|1) (\TT_1-\TT_{0})\bigr] &=& \sum_{\substack{n\in\N,\\ x_1\in V_1^\times, y_1\in V_2^\times}} -\log L(o,y_1x_1|1)\cdot n\cdot \P\left[\begin{array}{c}\mathbf{W}_1=y_1x_1,\\ \TT_1-\TT_0=n\end{array}\right]\\
&=&   \sum_{x_1\in V_1^\times, y_1\in V_2^\times} -\log L(o,y_1x_1|1)\cdot \sum_{n\in\N} n\cdot \P\left[\begin{array}{c}\mathbf{W}_1=y_1x_1,\\ \TT_1-\TT_0=n\end{array}\right]\\
&= & \sum_{x_1\in V_1^\times, y_1\in V_2^\times} -\log L(o,y_1x_1|1)\cdot \mathcal{E}'(y_1x_1|1),
\end{eqnarray*}
that is, $\E\bigl[  -\log L(o,\mathbf{W}_1|1) (\TT_1-\TT_{0})\bigr]$ varies real-analytically in $\underline{p}\in\mathcal{P}_d$ since the latter sum is a sum of finitely many products of real-analytic functions in $\underline{p}\in\mathcal{P}_d$.
\end{proof}

\begin{proof}[Proof of Corollary \ref{cor:CLT-3}.]
In view of  formula  (\ref{equ:variance-h}) and real-analyticity of $\E[\TT_1-\TT_0]$,
it suffices to prove that the mapping
$$
\mathcal{P}_d\ni (p_1,\ldots,p_d)\mapsto \E\Bigl[ \Bigl( -\log L(o,\mathbf{W}_1|1)  -(\TT_1-\TT_{0}) \cdot h\Bigr)^2\Bigr]
$$
varies real-analytically. By Lemma \ref{lem:logL-exp-mom}, 
$$
\E\bigl[-\log L(o,\mathbf{W}_1|1)\bigr]\quad \textrm{ and }\quad \E\bigl[-\log L(o,\mathbf{W}_1|1)(\TT_1-\TT_0)\bigr]
$$ 
vary real-analytically in $p_1,\ldots,p_d$; by Corollary \ref{cor:h-analytic}, $h$ also varies real-analytically, and we have seen in the proof of Theorem \ref{th:CLT} that $\E\bigl[(\TT_1-\TT_0)^2\bigr]$ varies also real-analytically. Following the same reasoning as in Lemma \ref{lem:logL-exp-mom}, we can see that the mapping
$$
\mathcal{P}_d\ni (p_1,\ldots,p_d)\mapsto \E\bigl[\bigl(-\log L(o,\mathbf{W}_1|1)\bigr)^2\bigr]
$$
varies real-analytically, since 
\begin{eqnarray*}
\E\bigl[\bigl(-\log L(o,\mathbf{W}_1|1)\bigr)^2\bigr] &=& \sum_{x_1\in V_1^\times, y_1\in V_2^\times}  -\log L(o,y_1x_1|1)^2\cdot \sum_{n\in\N} \P\left[\begin{array}{c}\mathbf{W}_1=y_1x_1,\\ \TT_1-\TT_0=n\end{array}\right]\\
&=&  \sum_{x_1\in V_1^\times, y_1\in V_2^\times}  -\log L(o,y_1x_1|1)^2\cdot \mathcal{E}(y_1x_1|1).
\end{eqnarray*}
The right hand side is a finite sum of products of real-analytic functions in $(p_1,\ldots,p_d)\in\mathcal{P}_d$.
%
%
%
This shows that $\sigma_h^2$ varies real-analytically in $(p_1,\ldots,p_d)\in\mathcal{P}_d$, which
  finishes the proof of Corollary \ref{cor:CLT-3}.
\end{proof}

\begin{rem}
In Corollary \ref{cor:CLT-3} we restricted ourselves to the case of free products of \textit{finite} graphs. This lies in the fact that the terms 
$$
-\log L(o,y_1x_1|1),\quad x_1\in V_1^\times, Y_1\in V_2^\times
$$ 
do also depend on the parameters $p_1\ldots,p_d$. In the case of \textit{infinite} free factors existence of a uniform neighbourhood $\mathcal{U}(\underline{p})$ in the proof of Lemma \ref{lem:logL-exp-mom} can not be guaranteed a-priori. For answering the question of real-analyticity of $\sigma_h^2$ in the infinite case a deeper analysis of the functions 
$$
\mathcal{P}_d \ni (p_1,\ldots,p_d)\mapsto -\log L(o,y_1x_1|1)
$$ 
would be necessary, which goes far beyond the scope of this article.
\end{rem}

\vfill

\pagebreak[4]
\appendix

\section{Proof of Proposition \ref{prop:Dk-iid}}
\label{app:proof-Dk-iid}

\begin{proof}[Proof of Proposition \ref{prop:Dk-iid}.(i)]
First, we show that $\DD_k$, $k\in\N$, has the same distribution as $\DD_1$. Recall that $\mathcal{W}_k$ denotes the support of $X_{\TT_k}$. Let be $n\in\N$. Then:
\begin{eqnarray*}
&& \P\bigl[  \DD_k=n\bigr] \\[1ex]
&=& \sum_{x\in \mathcal{W}_{k-1}, x_1\in V_1^\times, y_1\in V_2^\times: d(o,y_1x_1) =n} \P\bigl[ X_{\TT_{k-1}}=x, X_{\TT_k}=xy_1x_1\bigr] \\
&=& \sum_{\substack{l\in\N,\\ x\in \mathcal{W}_{k-1}}} \P\left[\begin{array}{c} X_{l-1}\notin C(x),\\ X_l=x\end{array}\right] \\
&&\quad 
\cdot \sum_{\substack{m\in\N,\\ x_1\in V_1^\times, y_1\in V_2^\times:\\ d(o,y_1x_1) =n}} \P_x\left[\begin{array}{c} X_{m-1}\notin C(xy_1x_1), \\ X_m=xy_1x_1,\\ \forall m'<m: X_{m'}\in C(x) \end{array}\right]\cdot (1-\xi_1)\\
&\stackrel{\textrm{Lemma \ref{lem:cone-probabilities}}}{=}&\underbrace{\sum_{l\in\N, x\in \mathcal{W}_{k-1}} \P\left[\begin{array}{c} X_{l-1}\notin C(x),\\ X_l=x\end{array}\right] \cdot (1-\xi_1)}_{=\P[\TT_{k-1}<\infty]=1}\\
&&\quad \cdot \sum_{\substack{m\in\N,\\ x_1\in V_1^\times, y_1\in V_2^\times:\\ d(o,y_1x_1)=n}} \P\left[\begin{array}{c} X_{m-1}\notin C(y_1x_1),\\ X_m=y_1x_1,\\ \forall m'<m: X_{m'}\notin V_1^\times\end{array}\right]\\
&=& \underbrace{\sum_{l\in\N, x\in \mathcal{W}_0} \P\left[\begin{array}{c} X_{l-1}\notin C(x),\\ X_l=x\end{array}\right]\cdot (1-\xi_1)}_{=\P[\TT_0<\infty]=1} \\
&&\quad \cdot \sum_{\substack{m\in\N,\\ x_1\in V_1^\times, y_1\in V_2^\times:\\ d(o,y_1x_1) =n}} \P\left[\begin{array}{c} X_{m-1}\notin C(y_1x_1),\\ X_m=y_1x_1,\\ \forall m'<m: X_{m'}\notin V_1^\times \end{array}\right]\\[1ex]
&\stackrel{\textrm{Lemma \ref{lem:cone-probabilities}}}{=}& \sum_{\substack{l\in\N,\\ x\in \mathcal{W}_0}} \P\left[\begin{array}{c} X_{l-1}\notin C(x),\\ X_l=x\end{array}\right] \\
&&\quad \cdot \sum_{\substack{m\in\N,\\ x_1\in V_1^\times, y_1\in V_2^\times:\\ d(o,y_1x_1) =n}} \P_x\left[\begin{array}{c} X_{m-1}\notin C(xy_1x_1),\\ X_m=xy_1x_1,\\ \forall m'<m: X_{m'}\in C(x) \end{array}\right]\cdot (1-\xi_1)\\[1ex]
&=& \sum_{\substack{x\in \mathcal{W}_{0}, x_1\in V_1^\times, y_1\in V_2^\times: \\ d(o,y_1x_1) =n}} \P\bigl[ X_{\TT_{0}}=x, X_{\TT_1}=xy_1x_1\bigr]\\
&=& \P\bigl[ \DD_1=n\bigr].
\end{eqnarray*}
This shows that the sequence $\bigl( \DD_k\bigr)_{k\in\N}$ is identically distributed.
\par
The proof of independence of $\bigl( \DD_k\bigr)_{k\in\N}$ follows the same reasoning as in Proposition \ref{prop:Tk-Tk-1-iid}, from which we take the notation of $w_0:=o,w_1,\ldots,w_k$. Let be $k\in\N$ and $n_1,\ldots,n_k\in\N$.  For $j\in\{1,\ldots,k\}$, we can write
\begin{eqnarray*}
\P[\DD_j=n_j] &=&\sum_{\substack{t\in\N,\\ w\in \mathcal{W}_{j-1}}} \P\left[\begin{array}{c} X_{t-1}\notin C(w),\\ X_t=w\end{array}\right] \\
&&\cdot 
\sum_{\substack{m_j\in \N,\\ x_j\in V_1^\times, y_j\in V_2^\times:\\ d(o,y_jx_j)= n_j}}  \P_{w}\left[\begin{array}{c} X_{m_j-1}\notin C(wy_jx_j), \\ X_{m_j}=wy_jx_j,\\ \forall m'<m_j: X_{m'}\in C(w) \end{array}\right] \cdot (1-\xi_1).
\end{eqnarray*}
Then independence is obtained as follows:
\begin{eqnarray*}
&& \P\bigl[  \DD_1=n_1,\ldots, \DD_k=n_k\bigr]\\[1ex]
&=& \sum_{\substack{x\in \mathcal{W}_0,\\ x_1,\ldots, x_k\in V_1^\times,\\ y_1,\ldots,y_k\in V_2^\times:\\ d(o,y_ix_i) = n_i}} \P\Bigl[X_{\TT_0}=x,X_{\TT_1}=xy_1x_1,\ldots,X_{\TT_k}=xy_1x_1\ldots y_kx_k\Bigr]\\
&=& \sum_{\substack{l\in\N,\\ x\in \mathcal{W}_0}} \P\left[\begin{array}{c} X_{l-1}\notin C(x),\\ X_l=x\end{array}\right] \\
&&\quad \cdot \sum_{\substack{m_1,\ldots, m_k\in\N,\\ x_1,\ldots, x_k\in V_1^\times,\\ y_1,\ldots,y_k\in V_2^\times:\\ d(o,y_ix_i) = n_i}} \prod_{j=1}^k \P_{xw_{j-1}}\left[\begin{array}{c} X_{m_j-1}\notin C(xw_j), \\ X_{m_j}=xw_j,\\ \forall m'<m_j: X_{m'}\in C(xw_{j-1}) \end{array}\right]\cdot (1-\xi_1)\\
&\stackrel{\textrm{Lemma \ref{lem:cone-probabilities}}}{=}& \sum_{\substack{l\in\N,\\ x\in \mathcal{W}_0}} \P\left[\begin{array}{c} X_{l-1}\notin C(x),\\ X_l=x\end{array}\right] \cdot
\sum_{\substack{m_1\in\N,\\ x_1\in V_1^\times, y_1\in V_2^\times:\\ d(o,y_1x_1)=n_1}} \P_x\left[\begin{array}{c} X_{m_1-1}\notin C(xy_1x_1), \\ X_{m_1}=xy_1x_1,\\ \forall m'<m_1: X_{m'}\in C(x) \end{array}\right]\\
&&\quad \cdot \sum_{\substack{m_2,\ldots,m_k\in\N,\\ x_2,\ldots, x_k\in V_1^\times,\\ y_2,\ldots,y_k\in V_2^\times:\\ d(o,y_ix_i)= n_i}} \prod_{j=2}^k 
\underbrace{\sum_{\substack{t\in\N,\\ w\in \mathcal{W}_{j-1}}} \P\left[\begin{array}{c} X_{t-1}\notin C(w),\\ X_t=w\end{array}\right] \cdot (1-\xi_1)}_{=\P[\TT_{j-1}<\infty]=1}\\
&& \quad\quad \cdot \P\left[\begin{array}{c} X_{m_j-1}\notin C(y_jx_j), \\ X_{m_j}=y_jx_j,\\ \forall m'<m_j: X_{m'}\notin V_1^\times  \end{array}\right]\cdot (1-\xi_1)
\end{eqnarray*}
\newpage
\begin{eqnarray*}
&\stackrel{\textrm{Lemma \ref{lem:cone-probabilities}}}{=}& \sum_{\substack{l\in\N,\\ x\in \mathcal{W}_0}} \P\left[\begin{array}{c} X_{l-1}\notin C(x),\\ X_l=x\end{array}\right] 
\cdot
\sum_{\substack{m_1\in\N,\\ x_1\in V_1^\times, y_1\in V_2^\times:\\ d(o,y_1x_1)=n_1}} \P_x\left[\begin{array}{c} X_{m_1-1}\notin C(xy_1x_1), \\ X_{m_1}=xy_1x_1,\\ \forall m'<m_1: X_{m'}\in C(x) \end{array}\right]\cdot (1-\xi_1)\\
&&\quad\quad \cdot \prod_{j=2}^k \Biggl( \sum_{\substack{m_j\in \N,\\ x_j\in V_1^\times, y_j\in V_2^\times:\\ d(o,y_jx_j)= n_j}}  
\sum_{\substack{t\in\N,\\ w\in \mathcal{W}_{j-1}}} \P\left[\begin{array}{c} X_{t-1}\notin C(w),\\ X_t=w\end{array}\right]  \\
&&\quad\quad\quad \cdot \P_{w}\left[\begin{array}{c} X_{m_j-1}\notin C(wy_jx_j), \\ X_{m_j}=wy_jx_j,\\ \forall m'<m_j: X_{m'}\in C(w) \end{array}\right] \cdot (1-\xi_1)\Biggr)\\
&=& \prod_{j=1}^k \P\bigl[  \DD_j=n_j\bigr].
\end{eqnarray*}
This proves independence of $\bigl( \DD_k\bigr)_{k\in\N}$.
\end{proof}


\begin{proof}[Proof of Proposition \ref{prop:Dk-iid}.(ii)]
The proof works completely analogously to the proof of Proposition \ref{prop:Tk-Tk-1-iid} and \ref{prop:Dk-iid}.(i); nonetheless, we sketch the proof for sake of completeness. Once again, we start with showing that $\widetilde \DD_k$, $k\in\N$, has the same distribution as $\widetilde \DD_1$. Observe that $\widetilde \DD_k$ is a discrete random variable with support
$$
\widetilde{\mathcal{W}}_k := \Bigl\{ z\in\mathbb{R} \,\Bigl|\, \P\bigl[\widetilde \DD_k=z\bigr]>0\Bigr\}.
$$
Recall also the definition of $\mathcal{W}_k$, the support of $X_{\TT_k}$.
For $z\in\widetilde{\mathcal{W}}_k$, we have then:
\begin{eqnarray*}
&& \P\bigl[ \widetilde \DD_k=z\bigr] \\[1ex]
&=& \sum_{\substack{m\in\N, x\in \mathcal{W}_{k-1},x_1\in V_1^\times,y_1\in V_2^\times:\\ d(o,y_1x_1)-m\cdot \lambda=z}} \P\bigl[ X_{\TT_{k-1}}=x, X_{\TT_k}=xy_1x_1, \TT_k-\TT_{k-1}=m\bigr]\\
&=& \sum_{\substack{l\in\N,\\ x\in \mathcal{W}_{k-1}}} \P\left[\begin{array}{c} X_{l-1}\notin C(x),\\ X_l=x\end{array}\right] \\
&&\quad 
\cdot \sum_{\substack{m\in\N,\\ x_1\in V_1^\times, y_1\in V_2^\times:\\ d(o,y_1x_1)-m\cdot\lambda =z}} \P_x\left[\begin{array}{c} X_{m-1}\notin C(xy_1x_1), \\ X_m=xy_1x_1,\\ \forall m'<m: X_{m'}\in C(x) \end{array}\right]\cdot (1-\xi_1)
\end{eqnarray*}
\begin{eqnarray*}
&\stackrel{\textrm{Lemma \ref{lem:cone-probabilities}}}{=}&\underbrace{\sum_{\substack{l\in\N,\\ x\in \mathcal{W}_{k-1}}} \P\left[\begin{array}{c} X_{l-1}\notin C(x),\\ X_l=x\end{array}\right] \cdot (1-\xi_1)}_{=\P[\TT_{k-1}<\infty]=1}\\
&&\quad \cdot \sum_{\substack{m\in\N,\\ x_1\in V_1^\times, y_1\in V_2^\times:\\ d(o,y_1x_1)-m\cdot \lambda=z}} \P\left[\begin{array}{c} X_{m-1}\notin C(y_1x_1),\\ X_m=y_1x_1,\\ \forall m'<m: X_{m'}\notin V_1^\times\end{array}\right]\\
&=&\underbrace{\sum_{\substack{l\in\N,\\ x\in \mathcal{W}_{0}}} \P\left[\begin{array}{c} X_{l-1}\notin C(x),\\ X_l=x\end{array}\right] \cdot (1-\xi_1)}_{=\P[\TT_{0}<\infty]=1}\\
&&\quad \cdot \sum_{\substack{m\in\N,\\ x_1\in V_1^\times, y_1\in V_2^\times:\\ d(o,y_1x_1)-m\cdot \lambda=z}} \P\left[\begin{array}{c} X_{m-1}\notin C(y_1x_1),\\ X_m=y_1x_1,\\ \forall m'<m: X_{m'}\notin V_1^\times\end{array}\right]\\
&\stackrel{\textrm{Lemma \ref{lem:cone-probabilities}}}{=}& \sum_{\substack{l\in\N,\\ x\in \mathcal{W}_0}} \P\left[\begin{array}{c} X_{l-1}\notin C(x),\\ X_l=x\end{array}\right] \\
&&\quad \cdot \sum_{\substack{m\in\N,\\ x_1\in V_1^\times, y_1\in V_2^\times:\\ d(o,y_1x_1)-m\cdot\lambda =z}} \P_x\left[\begin{array}{c} X_{m-1}\notin C(xy_1x_1),\\ X_m=xy_1x_1,\\ \forall m'<m: X_{m'}\in C(x) \end{array}\right]\cdot (1-\xi_1)\\[1ex]
&=& \sum_{\substack{m\in\N, x\in \mathcal{W}_{0},x_1\in V_1^\times,y_1\in V_2^\times:\\ d(o,y_1x_1)-m\cdot \lambda=z}} \P\bigl[ X_{\TT_{0}}=x, X_{\TT_1}=xy_1x_1, \TT_1-\TT_{0}=m\bigr]\\
&=& \P\bigl[\widetilde \DD_1=z\bigr].
\end{eqnarray*}
This shows that the sequence $\bigl(\widetilde \DD_k\bigr)_{k\in\N}$ is identically distributed.
\par
The proof of independence of $\bigl(\widetilde \DD_k\bigr)_{k\in\N}$ follows the same reasoning as in the proof of Proposition \ref{prop:Dk-iid}.(i);   we use again the notion of $w_0=o,w_1\ldots,w_k$. Let be $k\in\N$ and \mbox{$z_1,\ldots,z_k\in\widetilde{\mathcal{W}}_1$.}  Observe that, for $j\in\{1,\ldots,k\}$,
\begin{eqnarray*}
&&\P\bigl[\widetilde \DD_j=z_j\bigr] \\
&=&\sum_{\substack{t\in\N,\\ w\in \mathcal{W}_{j-1}}} \P\left[\begin{array}{c} X_{t-1}\notin C(w),\\ X_t=w\end{array}\right] \\
&&\quad \cdot \sum_{\substack{m_j\in \N,\\ x_j\in V_1^\times, y_j\in V_2^\times:\\ d(o,y_jx_j)-m_j\cdot \lambda = z_j}} \P_{w}\left[\begin{array}{c} X_{m_j-1}\notin C(wy_jx_j), \\ X_{m_j}=wy_jx_j,\\ \forall m'<m_j: X_{m'}\in C(w) \end{array}\right] \cdot (1-\xi_1).
\end{eqnarray*}
Then:
\begin{eqnarray*}
&& \P\Bigl[ \widetilde \DD_1=z_1,\ldots,\widetilde \DD_k=z_k\Bigr]\\[1ex]
&=& \sum_{\substack{m_1,\ldots, m_k\in\N,\\ x\in\mathcal{W}_0,\\ x_1,\ldots, x_k\in V_1^\times,\\ y_1,\ldots,y_k\in V_2^\times:\\ d(o,y_ix_i)-m_i\cdot \lambda = z_i}}\P\left[ \begin{array}{c}X_{\TT_0}=x,X_{\TT_1}=xy_1x_1,\ldots,X_{\TT_k}=xy_1x_1 \ldots y_kx_k,\\ \TT_1-\TT_0=m_1,\ldots, \TT_k-\TT_{k-1}=m_k\end{array}\right]
\end{eqnarray*}
\pagebreak[4]
\begin{eqnarray*}
&=& \sum_{\substack{l\in\N,\\ x\in \mathcal{W}_0}} \P\left[\begin{array}{c} X_{l-1}\notin C(x),\\ X_l=x\end{array}\right] \\
&&\quad \cdot \sum_{\substack{m_1,\ldots, m_k\in\N,\\ x_1,\ldots, x_k\in V_1^\times,\\ y_1,\ldots,y_k\in V_2^\times:\\ d(o,y_ix_i)-m_i\cdot \lambda = z_i}} \prod_{j=1}^k \P_{xw_{j-1}}\left[\begin{array}{c} X_{m_j-1}\notin C(xw_j), \\ X_{m_j}=xw_j,\\ \forall m'<m_j: X_{m'}\in C(xw_{j-1}) \end{array}\right]\cdot (1-\xi_1)\\
&\stackrel{\textrm{Lemma \ref{lem:cone-probabilities}}}{=}& \sum_{\substack{l\in\N,\\ x\in \mathcal{W}_0}} \P\left[\begin{array}{c} X_{l-1}\notin C(x),\\ X_l=x\end{array}\right] \cdot
\sum_{\substack{m_1\in\N,\\ x_1\in V_1^\times, y_1\in V_2^\times:\\ d(o,y_1x_1)-m_1\cdot \lambda=z_1}} \P_x\left[\begin{array}{c} X_{m_1-1}\notin C(xy_1x_1), \\ X_{m_1}=xy_1x_1,\\ \forall m'<m_1: X_{m'}\in C(x) \end{array}\right]\\
&&\quad \cdot \sum_{\substack{m_2,\ldots,m_k\in\N,\\ x_2,\ldots, x_k\in V_1^\times,\\ y_2,\ldots,y_k\in V_2^\times:\\ d(o,y_ix_i)-m_i\cdot \lambda = z_i}} \prod_{j=2}^k 
\underbrace{\sum_{\substack{t\in\N,\\ w\in \mathcal{W}_{j-1}}} \P\left[\begin{array}{c} X_{t-1}\notin C(w),\\ X_t=w\end{array}\right] \cdot (1-\xi_1)}_{=\P[\TT_{j-1}<\infty]=1}\\
&& \quad\quad \cdot \P\left[\begin{array}{c} X_{m_j-1}\notin C(y_jx_j), \\ X_{m_j}=y_jx_j,\\ \forall m'<m_j: X_{m'}\notin V_1^\times \end{array}\right]\cdot (1-\xi_1)\\
&\stackrel{\textrm{Lemma \ref{lem:cone-probabilities}}}{=}& \sum_{\substack{l\in\N,\\ x\in \mathcal{W}_0}} \P\left[\begin{array}{c} X_{l-1}\notin C(x),\\ X_l=x\end{array}\right] \\
&&\quad \cdot
\sum_{\substack{m_1\in\N,\\ x_1\in V_1^\times,y_1\in V_2^\times:\\ d(o,y_1x_1)-m_1\cdot \lambda =z_1}} \P_x\left[\begin{array}{c} X_{m_1-1}\notin C(xy_1x_1), \\ X_{m_1}=xy_1x_1,\\ \forall m'<m_1: X_{m'}\in C(x) \end{array}\right]\cdot (1-\xi_1)\\
&&\quad \cdot \prod_{j=2}^k \Biggl(\sum_{\substack{m_j\in \N,\\ x_j\in V_1^\times, y_j\in V_2^\times:\\ d(o,y_jx_j)-m_j\cdot \lambda = z_j}}  
\sum_{\substack{t\in\N,\\ w\in \mathcal{W}_{j-1}}} \P\left[\begin{array}{c} X_{t-1}\notin C(w),\\ X_t=w\end{array}\right] \\
&&\quad\quad \cdot \P_{w}\left[\begin{array}{c} X_{m_j-1}\notin C(wy_jx_j), \\ X_{m_j}=wy_jx_j,\\ \forall m'<m_j: X_{m'}\in C(w) \end{array}\right] \cdot (1-\xi_1)\Biggr)\\
&=& \prod_{j=1}^k \P\Bigl[ \widetilde \DD_j=z_j\Bigr].
\end{eqnarray*}
This proves independence of $\bigl(\widetilde \DD_k\bigr)_{k\in\N}$.
\end{proof}

\section*{Acknowledgement}

The author is grateful to both anonymous referees for their suggestions and hints regarding
content and exposition of the present article.

\bibliography{sn-bibliography}

@InProceedings{lalley:04,
author="Lalley, Steven P.",
OPTeditor="Maass, Alejandro
and Mart{\'i}nez, Servet
and Mart{\'i}n, Jaime San",
title="Algebraic Systems of Generating Functions and Return Probabilities for Random Walks",
booktitle="Dynamics and Randomness II",
year="2004",
publisher="Springer Netherlands",
address="Dordrecht",
pages="81--122",
isbn="978-1-4020-2469-6"
}

@Book{woess,
  author = 	 {Wolfgang Woess},
  ALTeditor = 	 {Wolfgang Woess},
  title = 	 {Random Walks on Infinite Graphs and Groups},
  publisher = 	 {Cambridge University Press},
  address = {Cambridge},
  year = 	 {2000}
  }

@Book{billingsley:99,
  author = 	 {P. Billingsley},
  title = 	 {Convergence of Probability Measures},
  publisher = 	 {Wiley},
  year = 	 {1999},
  OPTkey = 	 {},
  OPTvolume = 	 {},
  OPTnumber = 	 {},
  OPTseries = 	 {},
  address = 	 {New York},
  OPTedition = 	 {},
  OPTmonth = 	 {},
  OPTnote = 	 {},
  OPTannote = 	 {}
}

@PhdThesis{gilch,
  author = 	 {Lorenz A. Gilch},
  title = 	 {Rate of Escape of Random Walks},
  school = 	 {University of Technology Graz, Austria},
  year = 	 {2007},
  OPTkey = 	 {},
  OPTtype = 	 {},
  OPTaddress = 	 {},
  OPTmonth = 	 {},
  OPTnote = 	 {},
  OPTannote = 	 {}
}

@article{gilch:15,
  author = 	 {Lorenz Gilch},
  title = 	 {Asymptotic Entropy of Random Walks on Regular Languages over a Finite Alphabet},
  journal = 	 {Electron. J. Probab.},
  year = 	 {2016},
  volume = 	 {21},
  number = 	 {8},
  pages = 	 {1--42}
}

@Article{gilch:07,
  author = 	 {Lorenz A. Gilch},
  title = 	 {Rate Of Escape of Random Walks on Free Products},
  journal = 	 {J. Aust. Math. Soc.},
  year = 	 {2007},
  OPTkey = 	 {},
  volume = 	 {83},
  number = 	 {I},
  pages = 	 {31--54},
  OPTmonth = 	 {},
  OPTnote = 	 {},
  OPTannote = 	 {}
}

@article{candellero-gilch,
  author = 	 {E. Candellero and L. Gilch},
  title = 	 {Phase Transitions for Random Walk Asymptotics on Free Products of Groups},
  journal = 	 {Random Struct. Alg.},
  VOLUME={40},
  number={2},
  pages={150--181},
  OPTkey = 	 {},
  OPTmonth = 	 {},
  year = 	 {2009},
  OPTannote = 	 {}
}

@article{gilch:11,
  author = 	 {Lorenz A. Gilch},
  title = 	 {Asymptotic Entropy of Random Walks on Free Products},
  journal = 	 {Electron. J. Probab.},
  year = 	 {2011},
  volume = 	 {16},
  number = 	 {},
  pages = 	 {76--105}
}

@Article{gilch-ledrappier13,
  author = 	 {Lorenz Gilch and Francois Ledrappier},
  title = 	 {Regularity of the Drift and Entropy of Random Walks on Groups},
  journal = 	 {Publ. Mat. Urug.},
  year = 	 {2013},
  OPTkey = 	 {},
  volume = 	 {14},
  OPTnumber = 	 {},
  pages = 	 {147--158},
  OPTmonth = 	 {},
  OPTnote = 	 {},
  OPTannote = 	 {}
}

@Article{ledrappier12,
  author = 	 {F. Ledrappier},
  title = 	 {Analyticity of the entropy for some random walks},
  journal = 	 {Groups Geom. Dyn.},
  year = 	 {2012},
  OPTkey = 	 {},
  volume = 	 {6},
  OPTnumber = 	 {},
  pages = 	 {317--333},
  OPTmonth = 	 {},
  OPTnote = 	 {},
  OPTannote = 	 {}
}

@Article{ledrappier12-2,
  author = 	 {F. Ledrappier},
  title = 	 {Regularity of the entropy for random walks on hyperbolic groups},
  journal = 	 {Ann. Probab.},
  year = 	 {2013},
  OPTkey = 	 {},
  volume = 	 {41},
  number = 	 {5},
  pages = 	 {3582--3605},
  OPTmonth = 	 {},
  OPTnote = 	 {},
  OPTannote = 	 {}
}

@Article{mairesse1,
  author = 	 {Jean Mairesse and Frederic Math\'eus},
  title = 	 {Random walks on free products of cyclic groups},
  journal = 	 {J. London Math. Soc.},
  year = 	 {2007},
  OPTkey = 	 {},
  volume = 	 {75},
  number = 	 {1},
  pages = 	 {47--66},
  OPTmonth = 	 {},
  OPTnote = 	 {},
  OPTannote = 	 {}
}

@Article{sawyer78,
  author = 	 {Stanley Sawyer},
  title = 	 {Isotropic random walks in a tree},
  journal = 	 {Zeitschrift f. Wahrscheinlichkeitstheorie},
  year = 	 {1978},
  OPTkey = 	 {},
  volume = 	 {Verw. Geb. 42},
  OPTnumber = 	 {},
  pages = 	 {279--292},
  OPTmonth = 	 {},
  OPTnote = 	 {},
  OPTannote = 	 {}
}

@Article{cartwright-soardi,
  author = 	 {D.I. Cartwright and P.M. Soardi},
  title = 	 {Random walks on free products, quotients and amalgams},
  journal = 	 {Nagoya Math. J.},
  year = 	 {1986},
  OPTkey = 	 {},
  volume = 	 {102},
  OPTnumber = 	 {},
  pages = 	 {163--180},
  OPTmonth = 	 {},
  OPTnote = 	 {},
  OPTannote = 	 {}
}

@Article{woess3,
  author = 	 {Wolfgang Woess},
  title = 	 {Nearest neighbour random walks on free products of discrete groups},
  journal = 	 {Boll. Un. Mat. Ital.},
  year = 	 {1986},
  OPTkey = 	 {},
  volume = 	 {5-B},
  OPTnumber = 	 {},
  pages = 	 {961--982},
  OPTmonth = 	 {},
  OPTnote = 	 {},
  OPTannote = 	 {}
}

@Article{mathieu:15,
  author = 	 {Pierre Mathieu},
  title = 	 {Differentiating the entropy of random walks on hyperbolic groups},
  journal = 	 {Ann. Probab.},
  year = 	 {2015},
  OPTkey = 	 {},
  volume = 	 {43},
  number = 	 {1},
  pages = 	 {166--187},
  OPTmonth = 	 {},
  OPTnote = 	 {},
  OPTannote = 	 {}
}

@Article{gerl-woess,
  author = 	 {P. Gerl and W. Woess},
  title = 	 {Local Limits and Harmonic Functions for Nonisotropic Random Walks on Free Groups},
  journal = 	 {Probab. Theory Rel. Fields},
  year = 	 {1986},
  OPTkey = 	 {},
  volume = 	 {71},
  OPTnumber = 	 {},
  pages = 	 {341-355},
  OPTmonth = 	 {},
  OPTnote = 	 {},
  OPTannote = 	 {}
}

@Article{lalley93,
  author = 	 {S.P. Lalley},
  title = 	 {Finite range random walk on free groups and homogeneous trees},
  journal = 	 {Ann. Probab.},
  year = 	 {1993},
  OPTkey = 	 {},
  volume = 	 {21},
  number = 	 {4},
  pages = 	 {2087--2130},
  OPTmonth = 	 {},
  OPTnote = 	 {},
  OPTannote = 	 {}
}

@Article{haissinsky-mathieu-mueller12,
  author = 	 {P. Ha\"issinsky and P. Mathieu and S. M\"uller},
  title = 	 {Renewal theory for random walks on surface groups},
  journal = 	 {Ergodic Theory and Dynamical Systems},
  year = 	 {2018},
  OPTkey = 	 {},
  volume = 	 {38},
  number = 	 {1},
  pages = 	 {155--179},
  OPTmonth = 	 {},
  OPTnote = 	 {},
  OPTannote = 	 {}
}

@Article{gouezel:15,
  author = 	 {S. Gou\"ezel},
  title = 	 {Analyticity of the entropy and the escape rate of random walks in hyperbolic groups},
  journal = 	 {Discrete Analysis},
  year = 	 {2017},
  OPTkey = 	 {},
  volume = 	 {7},
  OPTnumber = 	 {},
  pages = 	 {1--37},
  OPTmonth = 	 {},
  OPTnote = 	 {},
  OPTannote = 	 {}
}

@inproceedings{stallings:71,
  title     = {Group Theory and Three-dimensional Manifolds},
  author    = {John P. Stallings},
  year      = {1971},
  booktitle = {Proceedings of the International Congress of Mathematicians (Nice, 1970)},
  volume    = {2},
  pages     = {73--79}
}

@Article{sidoravicius:18,
author = {Shi, Zhan and Sidoravicius, Vladas and Song, He and Wang, Longmin and Xiang, Kainan},
year = {2018},
OPTmonth = {05},
pages = {},
journal ={arXiv: https://arxiv.org/abs/1805.01611},
title = {On Spectral Radius of Biased Random Walks on Infinite Graphs}
}

@article{gilch:22,
title = {Range of random walks on free products},
journal = {Stochastic Processes and their Applications},
volume = {149},
pages = {369-403},
year = {2022},
issn = {0304-4149},
url = {https://www.sciencedirect.com/science/article/pii/S0304414922000588},
author = {Lorenz A. Gilch}
}

@Article{choi:26,
author = {Inhyeok Choi},
year = {2026},
volume = 	 {194},
number= {},
pages = {1055--1122},
journal ={Probab. Theory Relat. Fields},
title = {Contracting isometries and differentiability of the escape rate}
}

@Article{silva:26,
author = {Eduardo Silva},
year = {2026},
volume = 	 {837},
number= {},
pages = {131--183},
journal ={Journal für die reine und angewandte Mathematik},
title = {Continuity of asymptotic entropy on wreath products}
}

\end{document}